\documentclass[onefignum,onetabnum]{siamart220329}

\usepackage{graphics}

\usepackage{lipsum}
\usepackage{amsfonts}
\usepackage{graphicx}
\usepackage{epstopdf}
\usepackage{algorithmic}
\usepackage{subfigure}

\usepackage{multirow}
\usepackage{epsfig}
\usepackage{booktabs}
\usepackage{cases}

\usepackage{amsbsy,amsmath,latexsym,amsfonts, epsfig, color, amssymb, graphics, bm}
\usepackage{mathrsfs}
\usepackage{cancel}

\ifpdf
  \DeclareGraphicsExtensions{.eps,.pdf,.png,.jpg}
\else
  \DeclareGraphicsExtensions{.eps}
\fi

\newsiamremark{remark}{Remark}
\newsiamremark{hypothesis}{Hypothesis}
\crefname{hypothesis}{Hypothesis}{Hypotheses}
\newsiamthm{claim}{Claim}

\newtheorem{aassumption}{Assumption}
\newtheorem{fact}{Fact}

\headers{An inexact $q$-order regularized proximal Newton method}{R. Liu, S. Pan, and Y. Qian}

\title{An inexact $q$-order regularized proximal Newton method for nonconvex composite optimization\thanks{Submitted to the editors DATE.
\funding{This work was funded by the National Natural Science Foundation of China under projects No. 12371299.}}}

\author{Ruyu Liu
	\thanks{School of Mathematics, South China University of Technology (\email{maruyuliu@mail.scut.edu.cn}).}
	\and Shaohua Pan
	\thanks{School of Mathematics, South China University of Technology (\email{shhpan@scut.edu.cn}).}
	\and Yitian Qian
	\thanks{School of Mathematics, South China University of Technology (\email{qianyitian2014@163.com}).}
 }

\usepackage{amsopn}

\ifpdf
\hypersetup{
  pdftitle={An inexact $q$-order regularized proximal Newton method for nonconvex composite optimization},
  pdfauthor={R Liu, S. Pan and Y. Qian}
}
\fi

\begin{document}

\maketitle

\begin{abstract}
This paper concerns the composite problem of minimizing the sum of a twice continuously differentiable function $f$ and a nonsmooth convex function. For this class of nonconvex and nonsmooth problems, by leveraging a practical inexactness criterion and a novel selection strategy for iterates, we propose an inexact $q$-order regularized proximal Newton method for $q\in[2,3]$, which becomes an inexact cubic regularization (CR) method for $q=3$. We prove that the whole iterate sequence converges to a stationary point for the KL objective function; and when the objective function has the KL property of exponent $\theta\in(0,\frac{q-1}{q})$, the convergence has a local $Q$-superlinear rate of order $\frac{q-1}{\theta q}$. In particular, under a local H\"{o}lderian error bound of order $\gamma\in(\frac{1}{q-1},1]$ on a second-order stationary point set, we show that the iterate and objective value sequences converge to a second-order stationary point and a second-order stationary value, respectively, with a local $Q$-superlinear rate of order $\gamma(q\!-\!1)$, specified as the $Q$-quadratic rate for $q=3$ and $\gamma=1$. This is the first practical inexact CR method with $Q$-quadratic convergence rate for nonconvex composite optimization. We validate the efficiency of the CR method with ZeroFPR as the inner solver by applying it to composite optimization problems with highly nonlinear $f$.  
\end{abstract}

\begin{keywords}
 Nonconvex composite optimization; Inexact $q$-order regularized proximal Newton method; Cubic regularization method; KL property; Local H\"{o}lderian error bound
\end{keywords}

\begin{MSCcodes}
 90C26, 49M15, 90C55
\end{MSCcodes}

\section{Introduction}\label{sec1}

 Let $\mathbb{X}$ denote a finite dimensional real vector space endowed with the inner product $\langle\cdot,\cdot\rangle$ and its induced norm $\|\cdot\|$, and write $\overline{\mathbb{R}}\!:=(-\infty,\infty]$. We are interested in the following nonconvex and nonsmooth composite problem
 \begin{equation}\label{prob}
  \min_{x\in\mathbb{X}}F(x):=f(x)+g(x),
 \end{equation}
where $f\!:\mathbb{X}\to\overline{\mathbb{R}}$ and $g\!:\mathbb{X}\to\overline{\mathbb{R}}$ are the functions satisfying the following assumption:
\begin{aassumption}\label{ass1}
\begin{description}
 \item[(i)] $f$ is a proper lower semicontinuous (lsc) function that is twice continuously differentiable on an open set $\mathcal{O}\supset{\rm cl}\,({\rm dom}\,g)$;
		
 \item[(ii)] $g$ is a proper lsc convex function with ${\rm dom}\,g\ne\emptyset$, and $F$ is lower bounded.
\end{description} 
\end{aassumption}
 Model \eqref{prob} allows $g$ to be an indicator of a nonempty closed convex set. Furthermore, it covers the case that $g$ is weakly convex. Indeed, by recalling that $g$ is $\alpha$-weakly convex for some $\alpha>0$ if $g(\cdot)\!+\!({\alpha}/{2})\|\cdot\|^2$ is convex, the function $F$ can be rewritten as $F\!=\overline{f}+\!\overline{g}$ with $\overline{f}(\cdot)\!=f(\cdot)\!-\!(\alpha/2)\|\cdot\|^2$ and $\overline{g}(\cdot)\!=g(\cdot)\!+\!(\alpha/2)\|\cdot\|^2$. Obviously, $\overline{f}$ and $\overline{g}$ satisfy Assumption \ref{ass1}. Such a problem frequently appears in image processing \cite{Bonettini17,Chambolle16}, statistics \cite{HastieTW15}, machine learning \cite{Bottou18}, and financial engineering \cite{Zhou21}, where the smooth function $f$ represents a data fidelity term, and the nonsmooth function $g$ encodes some structure information on the ground truth such as sparsity and nonnegativity. 
\subsection{Related work}\label{sec1.1}

 A variety of methods have been developed for solving  problem \eqref{prob} such as block coordinate descent methods \cite{Tseng09}, variable metric proximal gradient methods \cite{Bonettini20,Bonettini23}, proximal Newton-type methods (see \cite{Kanzow21,LiuPanWY22}), those methods based on the forward-backward envelope of $F$ (see \cite{Themelis18,Ahookhosh21}), and so on. Here, we mainly review cubic regularization (CR) methods that are closely linked to this work. This class of methods is a regularized version of proximal Newton methods, and at each iteration minimizes the sum of a cubic regularized quadratic approximation of $f$ at the current iterate and the nonsmooth convex function $g$:
\begin{equation*}
 \min_{x\in\mathbb{X}}\, \vartheta_k(x)+({L_k}/{3})\|x-x^k\|^3+g(x),
\end{equation*}
where $L_k>0$ is a regularization parameter, and $\vartheta_k$ is the quadratic Taylor approximation of $f$ at the current iterate $x^k$ defined by 
\begin{equation}\label{varthetak}
 \vartheta_k(x):=f(x^k)+\langle\nabla\!f(x^k),x\!-\!x^k\rangle+\frac{1}{2}\langle\nabla^2\!f(x^k)(x\!-\!x^k),x\!-\!x^k\rangle\quad\ \forall x\in\mathbb{X}.
\end{equation}

 Early CR methods are tailored for unconstrained smooth optimization, i.e., problem \eqref{prob} with $g\equiv0$. Griewank \cite{Griewank81} first proposed a CR method for this class of problems and proved that any cluster point of the iterate sequence is a second-order critical point of $f$, i.e., a stationary point $x$ with $\nabla^2\!f(x)\succeq0$. Later, Nesterov and Polyak \cite{Nesterov06} provided a global efficiency estimate for a general class of smooth problems, and for three classes of specific smooth problems, under suitable conditions, they obtained the global complexity results and established the superlinear convergence rate of the objective value sequence. Their work triggers a burst of research on CR methods and tensor methods \cite{Schnabel91}. Cartis et al. \cite{Cartis11a} proposed adaptive CR methods by incorporating an adaptive strategy to tune the regularization parameter into the CR method and allowing the inexact computing of subproblems. They proved in \cite{Cartis11a} the full convergence and local Q-superlinear convergence rate of the iterate sequence under the local strong convexity of $f$, and achieved in \cite{Cartis11b} the global complexity bound to match that of \cite{Nesterov06}. Grapiglia and Nesterov \cite{Grapiglia17} extended the above complexity bound results to unconstrained smooth optimization of a twice differentiable $f$ with H\"{o}lder continuous Hessian. Yue et al. \cite{Yue19cubic} proved the full convergence and quadratic convergence rate of the iterate sequence for the CR method under a local Lipschitzian error bound condition, much weaker than the local strong convexity required in \cite{Nesterov06,Cartis11a} for the local quadratic convergence rate of the iterate sequence.

 Recently, Grapiglia and Nesterov \cite{Grapiglia19} extended the accelerated CR method \cite{Nesterov08} for unconstrained smooth convex optimization to problem \eqref{prob} with a twice differentiable convex $f$ whose Hessian is H\"{o}lder continuous, and achieved the global complexity bound of the objective values. For this class of problem \eqref{prob}, under an additional uniform convexity \footnote{A function $h\!:\mathbb{X}\to\overline{\mathbb{R}}$ is said to be uniformly convex of degree $\alpha\ge2$ over a convex set $C\subset{\rm dom}\,h$ if for some constant $\sigma>0$, $h(y)\ge h(x)+\langle \xi,y-x\rangle+\frac{\sigma}{\alpha}\|y-x\|^{\alpha}$ holds for all $x,y\in C$ and $\xi\in\partial h(x)$.} of $f$ with degree $\nu\!+\!2$, where $\nu$ is the order of H\"{o}lder continuity, Doikov and Nesterov \cite{Doikov21} established the linear convergence rate of objective value sequence of the CR method. Doikov and Nesterov \cite{Doikov20,Doikov22} also investigated (inexact) tensor methods for problem \eqref{prob} with convex $f$ by the $p$th-order derivatives of $f$ with $p\ge 2$, which reduce to an (inexact) CR method for $p=2$. They proposed in \cite{Doikov20} two dynamic strategies for choosing the computation accuracy of subproblems in terms of the number of iterations and the objective value difference of successive iterates, respectively, and for the latter scheme, achieved the global complexity of the objective values, and the linear convergence rate of objective value sequence under the uniform convexity of $F$ with degree $p+1$. When the inner accuracy becomes better, they also established superlinear convergence rate of order $(p\!+\!1)/2$ for the objective value sequence under the strong convexity of $F$. In \cite{Doikov22}, they proposed a regularized composite tensor method, and obtained the local $Q$-superlinear convergence of order $p/(\alpha\!-\!1)$ of the objective value sequence for $p>\alpha\!-\!1$ under the uniform convexity of $F$ with degree $\alpha\ge 2$ and the exact computation of subproblems. It is worth mentioning that Doikov and Nesterov \cite{Doikov22fc} also considered a fully composite convex problem, covering problem \eqref{prob} with convex $f$, and proved the linear convergence rate of objective value sequence under the uniform convexity of all component functions with degree $p+1$. When the uniform convexity is absent, they analyzed the global convergence of objective value sequence by requiring the outer function to be subhomogeneous.

 For nonconvex problem \eqref{prob} where $f$ is $p$-times continuously differentiable with the $p$th-order derivative being H\"{o}lder continuous of degree $\beta_p\in[0,1]$, Cartis et al. \cite{Cartis18} studied the worst-case complexity bound in terms of the first-order optimality conditions, and achieved the complexity bound $O(\epsilon^{-\frac{2+\beta_2}{1+\beta_2}})$ for the CR method. Necoara and Lupu \cite{Necoara21} proposed a general higher-order majorization-minimization algorithm framework for minimizing a nonconvex and nonsmooth function by its $p$ higher-order surrogate at each iterate, and proved the superlinear convergence rate of the objective value sequence under the KL property with exponent $\theta\in(0,\frac{1}{1+p})$. Later, Nabou and Necoara \cite{Nabou22} developed a general composite higher-order algorithmic framework for \eqref{prob} with $f=\varphi\circ G$ and nonconvex $g$ by a $p$ higher-order surrogate of $G$ at each iterate. When higher-order composite surrogates are sufficiently smooth, they derived the $R$-linear convergence rate of the objective value sequence under the KL property of $F$ with exponent $\theta\in(0,\frac{p}{1+p}]$. The CR method for problem \eqref{prob} precisely corresponds to the frameworks in \cite{Necoara21,Nabou22} with $p=2$.  

 From the above discussions, the existing superlinear convergence rate results of (inexact) CR methods are all established for objective value sequences under the uniform convexity of $F$ with degree $\alpha\ge 2$ \cite{Doikov22}, strong convexity of $F$ \cite{Doikov20}, or the KL property of $F$ with exponent $\theta\in(0,{1}/{3})$ \cite{Necoara21}. For nonconvex composite problem \eqref{prob}, there is no work to touch the local superlinear rate or even full convergence of iterate sequences of (inexact) CR methods except for \cite{QianPan22}. In \cite{QianPan22} the authors extended the iterative framework in \cite{Attouch13} for minimizing a nonconvex and nonsmooth function so that the generated sequence possesses a Q-superlinear convergence rate, and showed that the iterate sequence generated by a $q$-order regularized proximal Newton method with $q\in[2,3]$ (reducing to the CR method when $q=3$) for problem \eqref{prob} with a nonconvex $g$ falls within this framework. However, this method is impractical because it requires computing exactly a stationary point of a nonconvex and nonsmooth composite problem at each iteration. Though a practical inexact CR method was developed in \cite{Cartis18} for problem \eqref{prob} with $g$ being the indicator of a closed convex set, only the worst-case evaluation complexity in terms of the first-order optimality conditions was achieved. It is unclear whether its iterate sequence has the full convergence and/or local superlinear convergence rate. Thus, for nonconvex composite problem \eqref{prob}, it is still necessary to design implementable inexact CR methods and study the full convergence and fast local convergence rate of iterate sequences.
 
 Meanwhile, the uniform convexity of degree $\alpha\ge 2$ or strong convexity of objective functions is very restrictive and implies that  the associated optimization problems have a unique minimum. The local strong convexity of $F$ around a critical point is also restrictive, which implies that the critical point is an isolated local minimum of \eqref{prob}. Such a property is absent in many applications of interest; for example, the nonconvex quadratic minimization problem over a box constraint with the objective function of the form $F(x)=\frac{1}{2}x^{\top}Qx+c^{\top}x+\chi_{[l,u]}(x)$, where $Q$ is an $n\times n$ real symmetric matrix $Q$ with $\lambda_{\rm min}(Q)<0$ and $\chi_{[l,u]}(\cdot)$ is the indicator function over the box set $[l,u]\subset\mathbb{R}^n$. This nonconvex and nonsmooth problem does not possess local strong convexity around those non-minimum stationary points. The KL property of $F$ with exponent $\theta\in(0,{1}/{3})$ is also stringent because it is almost impossible to find such $F$ for \eqref{prob} though there are KL examples of exponent $1/2$ for $F$ by \cite{LiPong18,ZhouSo17}.
 Recall that Yue et al. \cite{Yue19cubic} established the local quadratic convergence rate of the iterate sequence of the CR method for unconstrained smooth optimization under a local Lipschitzian error bound. It is natural to ask whether a practical inexact CR method can be designed for \eqref{prob} to have a quadratic convergence rate under a similar local error bound. This work aims to provide an affirmative answer to the question.
\subsection{Main contribution}\label{sec1.2}

 We propose an inexact $q$-order regularized proximal Newton method with $q\in\![2,3]$ for problem \eqref{prob} by utilizing a practical inexactness criterion and a novel selection strategy for the iterates. When $q=3$, it reduces to an inexact CR method. Our contribution involves the following three aspects.

 {\bf(1)} For the proposed algorithm, we establish the full convergence of the iterate sequence under the KL property of $F$ by combining the sufficient decrease of the objective values with the KL inequality on a potential function skillfully. Based on the full convergence analysis and the proof of \cite[Theorem 3.2]{QianPan22}, under the KL property of $F$ with exponent $\theta\in\!(0,\frac{q-1}{q})$, we obtain the superlinear convergence rate of order $\frac{q-1}{\theta q}$ for the iterate and objective value sequences, which are consistent with those of \cite[Theorem 4.1 (ii)]{QianPan22} for the algorithm of only theoretical values. KL functions are ubiquitous in nonconvex and nonsmooth optimization by \cite[Section 4]{Attouch10}, and KL property of exponent $\theta\!\in(\frac{1}{2},\frac{q-1}{q})$ for $q\in\!(2,3]$ is weaker than that of exponent $1/2$. Thus, our convergence results have a wide range of potential applications. Also, the superlinear convergence result of the cost value sequence improves the one got with $p=2$ in \cite{Necoara21} for \eqref{prob}, a special case of the optimization problem considered in \cite{Necoara21}.  
 
 {\bf(2)} Under a local H\"{o}lderian error bound of order $\gamma\in\!(\frac{1}{q-1},1]$ on a second-order stationary point set that reduces to the ordinary one used in \cite{Griewank81,Yue19cubic} if $g\equiv 0$, we prove that the iterate sequence converges to a second-order stationary point with a local $Q$-superlinear convergence rate of order $\gamma(q\!-\!1)$, and the  convergence of the objective value sequence has the same superlinear rate under an additional lower second-order growth on $F$ around the second-order stationary point. This implies that the iterate and objective value sequences of our inexact CR method converge quadratically to a second-order stationary point and a second-order stationary value, respectively, under a local Lipschitzian error bound. To the best of our knowledge, this is the first practical inexact CR method for \eqref{prob} to have a quadratic convergence rate without the local strong convexity of $F$,  improving the superlinear convergence rate of order $4/3$ obtained in \cite{QianPan22} under the KL property of $F$ with  exponent $1/2$, and extending the work \cite{Yue19cubic} to nonconvex and nonsmooth composite optimization. It is worth emphasizing that our proof of the local quadratic convergence rate for $q=3$ is not a direct extension of the proofs in \cite{QianPan22,Yue19cubic}. The local error bound condition is much weaker than the strong convexity required in \cite{Doikov22}, and is basically the same as the one used in \cite{Yue19cubic} when $g\equiv 0$. For the aforementioned quadratic program, though it is not locally strongly convex around any non-minimum stationary points, the local Lipschitzian error bound on the set of stationary points and the KL property of $F$ with exponent $1/2$ are implied by the local upper Lipschitz of $(\partial F)^{-1}$ due to \cite[Proposition 1]{Robinson81}.  We also derive a global complexity bound $O(\epsilon^{-\frac{q}{q-1}})$ in terms of the first-order optimality conditions with the strict continuity of $\nabla^2\!f$ on a compact convex set containing the iterate sequence. The complexity bound for $q=3$ matches that of the CR method \cite{Nesterov06} and its variants \cite{Cartis11b} for unconstrained smooth optimization, and agrees with the one in \cite[Theorem 3.12]{Cartis18} if $g=\chi_{C}$ for a closed convex set $C\subset\mathbb{X}$.
 
 {\bf(3)} We apply the proposed algorithm armed with ZeroFPR \cite{Themelis18} as the solver of subproblems to three classes of problem \eqref{prob}, and compare its performance with that of ZeroFPR directly applied to these problems and IRPNM proposed in \cite{LiuPanWY22}. Numerical results indicate that our algorithm has the performance matching theoretical results. For the $\ell_1$-norm regularized logisitic regression, our CR method returns the optimal values and KKT residuals with better accuracy within $54$ iterations except for data ``arcene''. For the nonconvex $\ell_1$-norm regularized Student's t-regression, it yields second-order stationary points with the same objective values and comparable KKT residuals within less iterations; and for the portfolio decision with higher moments, it generates the comparable objective values and better KKT residuals within comparable running time with IRPNM and much less running time than ZeroFPR.  
\subsection{Notation} 
 Throughout this paper, a capital hollow letter such as $\mathbb{X}$ or $\mathbb{Y}$ denotes a finite dimensional real vector space endowed with the inner product $\langle\cdot,\cdot\rangle$ and its induced norm $\|\cdot\|$.
For a closed set $C\subseteq\mathbb{X}$, $\chi_{C}$ denotes the indicator function of $C$, i.e., $\chi_{C}(x)=0$ if $x\in C$, otherwise $\chi_{C}(x)=\infty$; $\Pi_{C}$ represents the projection operator onto $C$ which may be multi-valued if $C$ is nonconvex; and ${\rm dist}(x,C)$ denotes the distance of a vector $x\in\mathbb{X}$ to the set $C$ with respect to the norm $\|\cdot\|$. For a vector $x\in\mathbb{X}$, $\mathbb{B}(x,\delta)$ denotes the closed ball centered at $x$ with radius $\delta>0$. 
An extended real-valued $h\!:\mathbb{X}\to\overline{\mathbb{R}}$ is proper if its domain ${\rm dom}\,h:=\{x\in\mathbb{X}\ |\ h(x)<\infty\}$ is nonempty. 
For a closed proper function $h\!:\mathbb{X}\to\overline{\mathbb{R}}$, $\partial h(x)$ denote its basic (also known as the Morduhovich) subdifferential at $x$; the notation $[\alpha<h<\beta]$ for real numbers $\alpha<\beta$ represents the set $\{x\in\mathbb{X}\,|\,\alpha<h(x)<\beta\}$; and ${\rm prox}_{\!\alpha h}$ denotes its proximal mapping associated to parameter $\alpha>0$, defined as 
${\rm prox}_{\!\alpha h}(x)\!:=\!\mathop{\arg\min}_{z\in\mathbb{X}}
\big\{\frac{1}{2\alpha}\|z\!-\!x\|^2+h(z)\big\}$. When $h$ is convex, the mapping ${\rm prox}_{\!\alpha h}(x)$ is nonexpansive, i.e., for any $x,y\in\mathbb{X}$, $\|{\rm prox}_{\!\alpha h}(x)-{\rm prox}_{\!\alpha h}(y)\|\le\|x\!-\!y\|$. In the sequel, we just write ${\rm prox}_{h}$ for ${\rm prox}_{1h}$.
\section{Preliminaries}\label{sec2}
We first recall the strict continuity and subdifferential continuity of a proper function $h\!:\mathbb{X}\to\overline{\mathbb{R}}$. 
By \cite[Definition 9.1]{RW98}, $h$ is strictly continuous at $\overline{x}\in{\rm dom}\,h$ if there exist a neighborhood $V$ of $\overline{x}$ and a constant $\kappa\ge 0$ such that for any $x',x''\in V$, 
$|h(x')-h(x'')|\le \kappa\|x'-x''\|$; and it is strictly continuous at $\overline{x}$ relative to ${\rm dom}\,h$ if there exists a neighborhood $V$ of $\overline{x}$ such that $h$ is Lipschitz continuous on ${\rm dom}\,h\cap V$. By \cite[Definition 13.28]{RW98}, the function $h$ is subdifferentially continuous at $\overline{x}$ for $\overline{v}$ if $\overline{v}\in\partial h(\overline{x})$ and, whenever $(x^k,v^k)\xrightarrow{{\rm gph}\,\partial h}(\overline{x},\overline{v})$, one has $h(x^k)\to h(\overline{x})$; and when this holds for all $\overline{v}\in\partial h(\overline{x})$, it is said to be subdifferentially continuous at $\overline{x}$.
\subsection{Stationary points of problem \eqref{prob}}\label{sec2.1}

 A vector $x\in\mathbb{X}$ is called a critical point of $F$ or a stationary point of \eqref{prob} if $0\in\partial F(x)$, and denote by $\mathcal{S}^*$ the set of stationary points of \eqref{prob}. Define the residual function of problem \eqref{prob} by 
 \begin{equation}\label{Rmap}
  r(x):=\|R(x)\|\ \ {\rm with} \ \ 
  R(x)\!:=x\!-\!{\rm prox}_g(x\!-\!\nabla\!f(x))\quad{\rm for}\ x\in\mathbb{X}.
 \end{equation}
 By the continuous differentiability of $f$ on $\mathcal{O}\supset{\rm cl}\,({\rm dom}\,g)$ and the convexity of $g$, it is easy to verify that $\mathcal{S}^*=\{x\in\mathbb{X}\,|\,r(x)=0\}$, and furthermore, $\mathcal{S}^*$ coincides with the set of $L$-stationary points as well as that of $D$-stationary points \footnote{We call $x\in{\rm dom}\,g$ an $L$-stationary point if there is $L>0$ such that $x={\rm prox}_{L^{-1}g}(x-L^{-1}\nabla\!f(x))$, and a $D$-stationary point if $\langle \nabla\!f(x),d\rangle+g'(x;d)\ge 0$ for all $d\in\mathbb{X}$.}. A stationary point $x\in\mathcal{S}^*$ is called a second-order stationary point of \eqref{prob} if for any $z\in\mathbb{X}$, 
 \begin{equation}\label{SOSS}
  \langle {\rm prox}_{g}(z)-x,\nabla^2\!f(x)({\rm prox}_{g}(z)-x)\rangle\ge 0,
 \end{equation}
 and denote by $\mathcal{X}^*$ the set of second-order stationary points. When $g\equiv 0$, this reduces to the one used in \cite{Yue19cubic}. When $g\not\equiv 0$ and $f$ is convex, $\mathcal{X}^*=\{x\in\mathcal{S}^*\,|\,\nabla^2\!f(x)\succeq 0\}=\mathcal{S}^*$. When $g\not\equiv 0$ and $f$ is nonconvex, the inclusion $\mathcal{X}^*\supset\{x\in\mathcal{S}^*\,|\,\nabla^2\!f(x)\succeq 0\}$ is generally strict because for a given $x\in\mathcal{S}^*$ the set ${\rm prox}_{g}(\mathbb{X})-x$ is much smaller than $\mathbb{X}$; for example, if $g=\chi_{C}$ for a closed convex set $C\subset\mathbb{X}$, ${\rm prox}_{g}(\mathbb{X})-x\subset C-x$. 
\subsection{Metric subregularity of residual mapping}\label{sec2.2}

 The metric $\gamma$-subregularity of a multifunction was used to analyze the convergence rate of proximal point algorithm for seeking a root to a maximal monotone operator \cite{LiMor12}, and the local superlinear convergence rates of proximal Newton-type methods for composite optimization problems \cite{Mordu23,LiuPanWY22}. Its formal definition is stated as follows.
\begin{definition}\label{Def2.1}(see \cite[Definition 3.1]{LiMor12})
 Let $\mathcal{F}\!:\mathbb{X}\rightrightarrows\mathbb{X}$ be a multifunction. Consider any point $(\overline{x},\overline{y})\in{\rm gph}\,\mathcal{F}$, the graph of $\mathcal{F}$. For a given $\gamma>0$, we say that $\mathcal{F}$ is (metrically) $\gamma$-subregular at $(\overline{x},\overline{y})$ if there exist $\kappa>0$ and $\delta>0$ such that 
 \begin{equation}\label{q-subregular}
 {\rm dist}(x,\mathcal{F}^{-1}(\overline{y}))\le\kappa[{\rm dist}(\overline{y},\mathcal{F}(x))]^{\gamma}\quad{\rm for\ all}\ x\in\mathbb{B}(\overline{x},\delta).
 \end{equation}
 When $\gamma=1$, this property is called the (metric) subregularity of $\mathcal{F}$ at $(\overline{x},\overline{y})$. 
\end{definition}

 Obviously, if $\mathcal{F}$ is subregular at $(\overline{x},\overline{y})\in{\rm gph}\,\mathcal{F}$, it is $\gamma$-subregular at this point for any $\gamma\in(0,1]$. For the residual mapping $R$ in \eqref{Rmap}, its $\gamma$-subregularity at $(\overline{x},0)$ with $\overline{x}\in\mathcal{S}^*$ for $\gamma\in(0,1]$ requires the existence of $\kappa>0$ and $\delta>0$ such that 
\begin{equation}\label{ebound-MS}
 {\rm dist}(x,\mathcal{S}^*)\le\kappa[r(x)]^{\gamma}\quad{\rm for\ all}\ x\in\mathbb{B}(\overline{x},\delta). 
\end{equation}
 This is the local H\"olderian error bound of order $\gamma$ on $\mathcal{S}^*$ at $\overline{x}$, and is a little weaker than the one on $\mathcal{X}^*$ at $\overline{x}\in\mathcal{X}^*$. Some checkable conditions for the subregularity of $R$ are derived by using the results of \cite{Bai19}; see \url{https://arxiv.org/abs/2311.06871v3}. 
\subsection{Kurdyka-{\L}ojasiewicz property}\label{sec2.3}

To recall the KL property, for every $\varpi>0$, denote $\Upsilon_{\!\varpi}$ by the set of continuous concave functions $\varphi\!:[0,\varpi)\to\mathbb{R}_{+}$ that are continuously differentiable on $(0,\varpi)$ with $\varphi(0)=0$ and $\varphi'(s)>0$ for all $s\in(0,\varpi)$.
\begin{definition}\label{KL-def}
 A proper function $h\!:\mathbb{X}\to\overline{\mathbb{R}}$ is said to have the KL property at a point $\overline{x}\in{\rm dom}\,\partial h$ if there exist $\delta>0,\varpi\in(0,\infty]$ and $\varphi\in\Upsilon_{\!\varpi}$ such that 
 \[
   \varphi'(h(x)\!-\!h(\overline{x})){\rm dist}(0,\partial h(x))\ge 1
 \]
 for all $x\in\mathbb{B}(\overline{x},\delta)\cap[h(\overline{x})<h<h(\overline{x})+\varpi]$; and it is said to have the KL property of exponent $\theta\in[0,1)$ at $\overline{x}$ if there exist $c>0,\,\delta>0$ and $\varpi\in(0,\infty]$ such that  
 \[
   c(1\!-\!\theta)\,{\rm dist}(0,\partial h(x))\ge (h(x)\!-\!h(\overline{x}))^{\theta}
 \]
 for all $x\in\mathbb{B}(\overline{x},\delta)\cap[h(\overline{x})<h<h(\overline{x})+\varpi]$. 
 If $h$ has the KL property (of exponent $\theta$) at each point of ${\rm dom}\,\partial h$, it is called a KL function (of exponent $\theta$).
\end{definition}

 By \cite[Lemma 2.1]{Attouch10}, to show that a proper lsc function has the KL property (of exponent $\theta\in[0,1)$), it suffices to check if the property holds at its critical points. The discussions in \cite[Section 2.2]{LiuPanWY22} show that the KL property of $h$ with exponent $\theta\in(0,1)$ has a close relation with the $\gamma$-subregularity of its subdifferential mapping $\partial h$ for $\gamma\in(0,1]$. As discussed in \cite[Section 4]{Attouch10}, the KL property is ubiquitous and the functions definable in an o-minimal structure over the real field admit this property. 
 \section{Inexact $q$-order regularized proximal Newton method}\label{sec3}
 
Let $x^k$ be the current iterate. When $\nabla^2\!f$ is H\"older calm with exponent $q\!-\!2$ at $x^k$, i.e., there exist $\overline{L}_k>0$ and $\delta_k>0$ such that $\|\nabla^2\!f(z)\!-\!\nabla^2\!f(x^k)\|\le \overline{L}_k\|z-x^k\|^{q-2}$ for all $z\in\mathbb{B}(x^k,\delta_k)$, it is easy to verify that $|f(z)\!-\!\vartheta_k(z)|\le({\overline{L}_k}/{q})\|z\!-\!x^k\|^q$ for all $z\in\mathbb{B}(x^k,\delta_k)$. This means that the function $z\mapsto \vartheta_k(z)+(L_k/q)\|z-x^k\|^q+g(z)$ with $L_k\ge\overline{L}_k$ is a local majorization of $F$ at $x^k$. Based on this, our method first constructs a local majorization of $F$ at the current iterate $x^k$ by adding the $q$-order regularization term $({L_{k}}/{q})\|\cdot-x^k\|^{q}$ to the local quadratic approximation, and then computes inexactly 
\begin{equation}\label{subprobk}
 \min_{x\in\mathbb{X}}\Theta_k(x):=f_k(x)+g(x)\ \ {\rm with}\ f_k(x)\!:=\vartheta_k(x)+({L_k}/q)\|x\!-\!x^k\|^q
\end{equation}
to seek an inexact stationary point $y^k$ satisfying the following conditions
\begin{subnumcases}{}\label{inexactk}
 r_k(y^k)\le\varrho L_k\|y^k\!-\!x^k\|^{q-1}\ {\rm for}\ \varrho\in(0,1)\ {\rm with}\ \Theta_k(y^k)\le\Theta_k(x^k),\\
	\label{descentk}
F(y^k)\le F(x^k)-\sigma(L_k/q)\|y^k\!-\!x^k\|^{q}\ \ {\rm for}\ \sigma\in(0,1),
\end{subnumcases}
where $r_k\!:\mathbb{X}\to\mathbb{R}$ is the KKT residual function of subproblem \eqref{subprobk} defined by
\begin{equation}\label{Rk-def}
 r_{k}(x)\!:=\|R_{k}(x)\|\ \ {\rm with}\ \ R_{k}(x)\!:=x-{\rm prox}_g\big(x-\nabla\!f_{k}(x)\big)\ {\rm for}\ x\in\mathbb{X}.
\end{equation}
The $q$-order regularization term $({L_k}/{q})\|\cdot-x^k\|^q$ accounts for $f_k$ being a better approximation than $\vartheta_k$ to $f$ at $x^k$. Although problem \eqref{subprobk} is still a nonconvex composite optimization one, it is much simpler than the original problem \eqref{prob} because its nonconvexity is only caused by the quadratic term. The above condition \eqref{inexactk} aims at seeking an approximate stationary point $y^k$ of subproblem \eqref{subprobk} with its objective value no worse than the current one $\Theta_k(x^k)$, while condition \eqref{descentk} requires the approximate stationary point $y^k$ to have an objective value $F(y^k)$ lower than the current one of problem \eqref{prob}. The computation of subproblem \eqref{subprobk} with an appropriate $L_k$ is not as difficult as it seems to be. Indeed, at the $k$th iteration, by the expression of $f_k$, we have $\nabla^2\!f_k(y)\succeq \nabla^2\!f(x^k)+L_k\|y-x^k\|^{q-2}\mathcal{I}$ for $y\in\mathbb{X}$. Hence, as long as $x^k$ is not a stationary point of \eqref{subprobk}, for an appropriately large $L_k$, it is possible for $f_k$ and then $\Theta_k$ to be strongly convex in a neighborhood of the stationary point set of \eqref{subprobk}. In addition, for an appropriately large $L_k$, the function $\Theta_k$ is coercive, so subproblem \eqref{subprobk} has an optimal solution, which by Lemma \ref{well-def} later ensures the existence of the desired $y^k$. Motivated by this, our method at each iteration employs a line-search procedure to capture an appropriate $L_k$ and a desired $y^k$ simultaneously.  

 With $y^k$, the next iterate $x^{k+1}$ is set to be either $y^k$ or $y^k\!-\!R_k(y^k)$ depending on which one has a lower objective value. Note that $y^k\!-\!R_k(y^k)={\rm prox}_g(y^k\!-\!\nabla\!f_k(y^k))$ is a potential stationary point of \eqref{prob} because it can be viewed as an approximate iterate produced by applying the iterative formula $x^{k+1}={\rm prox}_g(x^k\!-\!\nabla\!f(x^k))$. Hence, such a selection strategy takes into account the objective values of potential stationary points. As a consequence, it is possible to gain a sufficient decrease of the objective values as well as guarantee the convergence of the iterate sequence. A similar selection strategy was adopted in the variable metric inexact proximal gradient method \cite{Bonettini17,Bonettini20} and the inexact proximal Newton method \cite{LiuPanWY22} for nonconvex composite optimization, but it is done between $y^k$ and a point on the line segment joining $x^k$ and $y^k$, and does not involve the information on the objective values of potential stationary points.   

 The iteration process of our inexact $q$-order regularized proximal Newton method is described as follows, where for each $k\in\mathbb{N}$ and $j\in\mathbb{N}$, $r_{k,j}$ denotes the KKT residual of subproblem \eqref{subprobj} defined as in \eqref{Rk-def} except that $f_k$ is replaced with $f_{k,j}$, and $f_{k,j}$ is the function defined by $f_{k,j}(x):=\vartheta_{k}(x)+(L_{k,j}/q)\|x-x^k\|^q$ for $x\in\mathbb{X}$.
\begin{algorithm}[H]
 \caption{\label{qregPNT}{\bf (Inexact $q$-order regularized proximal Newton method)}}
 \textbf{Input:} $q\in[2,3],0<\!L_{m}<L_{\!M},\varrho\in(0,1),\sigma\in(0,1),\tau>1$ and $x^0\in{\rm dom}\,g$.
	
 \noindent
 \textbf{For} $k=0,1,2,\ldots$ \textbf{do}
 \begin{enumerate}		
 \item Choose $L_{k,0}\in[L_{m},L_{\!M}]$ and $L_{k,0}>\max(0,-\lambda_{\min}(\nabla^2f(x^k)))$ if $q=2$.
		
 \item \textbf{For} $j=0,1,2,\ldots$ \textbf{do}
		\begin{itemize}
			\item[(2a)] Compute an inexact stationary point $y^{k,j}$ of the composite problem 
			\begin{equation}\label{subprobj}
				\min_{x\in\mathbb{X}}\ \Theta_{k,j}(x)\!:=f_{k,j}(x)+g(x)
			\end{equation}
			such that $y^{k,j}$ satisfies the following inexactness criterion
	\begin{equation}\label{inexact-cond}
	   r_{k,j}(y^{k,j})\le \varrho L_{k,j}\|y^{k,j}\!-x^k\|^{q-1}\ \ {\rm with}\ \ \Theta_{k,j}(y^{k,j})\le\Theta_{k,j}(x^k).
	\end{equation}
			
	     \item[(2b)] If $ F(y^{k,j})\le F(x^k)-(\sigma L_{k,j}/q)\|y^{k,j}-x^k\|^q$, then set $j_k=j$ and go to step 3; else let $L_{k,j+1}=\tau L_{k,j}$ and go to step (2a).
	   \end{itemize}
	\textbf{end (For)}
		
 \item Write $L_k\!:=L_{k,j_k},\,y^k\!:=y^{k,j_k}$ and $v^k\!:=R_{k,j_k}(y^k)$. Let 
		\begin{equation}\label{new-iter}
			x^{k+1}\!:=\left\{\begin{array}{cl}
				y^k &{\rm if}\ F(y^k)<F(y^k\!-\!v^k),\\
				y^k\!-\!v^k &{\rm  otherwise}.
			\end{array}\right.
		\end{equation}
         and go to step 1.

\end{enumerate}
\textbf{end (For)}
\end{algorithm} 
\begin{remark}\label{remark-alg}
 {\bf (a)} Recall that $r_{k,j}(y^{k,j})=0$ if and only if $y^{k,j}$ is a stationary point of \eqref{subprobj}. The inexactness criterion \eqref{inexact-cond} along with $\Theta_{k,j}(x^k)=F(x^k)$ means that step (2a) of Algorithm \ref{qregPNT} is searching for an inexact stationary point of \eqref{subprobj} with its objective value no worse than the current one of \eqref{prob}. Obviously, any efficient algorithms developed recently for nonconvex composite optimization such as ZeroFPR \cite{Themelis18} can be used as an inner solver.  Together with the checkability of the inexactness criterion \eqref{inexact-cond}, Algorithm \ref{qregPNT} is a practical $q$-order regularized proximal Newton method. This is the most significant difference from the one proposed in \cite{QianPan22}, which requires computing an exact stationary point of subproblem at each iteration. In addition, Algorithm \ref{qregPNT} involves a novel selection strategy for the next iterate $x^{k+1}$ which, as will be shown in Section \ref{sec4}, ensures that the iterate sequence generated by Algorithm \ref{qregPNT} is convergent. When $q=2$, Algorithm \ref{qregPNT} reduces to a regularized proximal Newton method for solving problem \eqref{prob}, but unlike the one developed in \cite{LiuPanWY22} its subproblems are not necessarily convex, and the adopted inexactness criterion and selection strategy on the next iterate are also different from those used in \cite{LiuPanWY22}.

 \noindent
 {\bf (b)} Algorithm \ref{inexact-cond} is well-defined by Lemma \ref{well-def} later. When $x^{k+1}\!=x^k$ occurs, $x^k$ is a stationary point of \eqref{prob} because now $x^k=y^k$ holds and the inexactness criterion \eqref{inexact-cond} implies that $r(x^k)=r_{k}(y^k)=0$. Indeed, from \eqref{new-iter}, either $x^{k+1}=y^k$ or $x^{k+1}=y^k-v^k$. For the former, $x^k=y^k$ automatically holds; while for the latter, $x^k=y^k-v^k$, which by step (2b) implies that 
 $\sigma(L_k/q)\|y^k\!-\!x^k\|^q\le F(x^k)\!-\!F(y^k)\le F(x^k)\!-\!F(y^k\!-\!v^k)=0$, so $x^k=y^k$ holds by $L_k\ge L_{m}>0$. In view of this, we suggest that $r(x^{k+1})\le\epsilon$ or $\|x^{k+1}\!-\!x^k\|\le\epsilon$ for an accuracy tolerance $\epsilon\in(0,1)$ as the stopping criterion of Algorithm \ref{qregPNT}. 
 
 \noindent
 {\bf (c)} For the smooth unconstrained problem $\min_{x\in\mathbb{R}^n}\!f(x)$, we claim that as $k$ is large enough, our inexactness condition \eqref{inexactk} with $g\equiv 0$ and $q=3$ is weaker than 
 \begin{equation}\label{Cartis-cond1}
  f_k(y^k)\le f_k(y^k_C)\ {\rm and}\  
  \|\nabla\!f_k(y^k)\|\le\kappa_{\theta}\min(1,\|y^k\!-\!x^k\|)\|\nabla\!f(x^k)\|\ {\rm for}\ \kappa_{\theta}\!\in(0,1),
 \end{equation}
 the one adopted by Cartis et al. \cite{Cartis11a,Cartis11b} for the adaptive cubic regularization methods with $m_k\equiv f_k$, where $y^k_C=x^k\!-\!\alpha_k^C\nabla\!f(x^k)$ with  $\alpha_k^C=\mathop{\arg\min}_{\alpha\ge0}m_k(x^k\!-\!\alpha\nabla\!f(x^k))$. Indeed, from the first inequality of \eqref{Cartis-cond1} and the definition of $\alpha_{k}^{C}$, for each $k\in\mathbb{N}$,
 \[
   f_k(y^k)\le f_k(y_{C}^k)\le f_k(x^k),
 \]
 which along with $g\equiv 0$ implies that $\Theta_k(y^k)\le\Theta_k(x^k)$. While from the second inequality of \eqref{Cartis-cond1}, we have $\|\nabla f_k(y^k)\|\le\kappa_{\theta}\|\nabla\!f(x^k)\|$, which in turn implies that  
 \[
  (1\!-\!\kappa_{\theta})\|\nabla\!f(x^k)\|\le\|\nabla\!f(x^k)\!-\!\nabla f_k(y^k)\|\le\|\nabla^2\!f(x^k)\|\|y^k\!-\!x^k\|+L_k\|y^k\!-\!x^k\|^2. 
\]
Along with the boundedness of $\{\nabla^2\!f(x^k)\}_{k\in\mathbb{N}}$ and $\{L_k\}_{k\in\mathbb{N}}$ and $\lim_{k\to\infty}\|y^k\!-\! x^k\|=0$ (see Proposition \ref{prop2-xk} later),  
it follows that $\|\nabla\!f(x^k)\|\le \widetilde{c}\|y^k\!-\!x^k\|$ with some $\widetilde{c}>0$ for all $k$ large enough, which by the second inequality of \eqref{Cartis-cond1} means that $\|\nabla\! f_k(y^k)\|\le\kappa_{\theta}\widetilde{c}\|y^k\!-\!x^k\|^2$, i.e., the first inequality of \eqref{inexactk} holds. Thus, the claimed relation holds. By following the same arguments, we conclude that as $k$ is sufficiently large, our inexactness condition \eqref{inexactk} with $g\equiv 0$ and $q=5/2$ is also weaker than 
\begin{equation*}
  f_k(y^k)\le f_k(y^k_C)\ \ {\rm and}\ \  
 \|\nabla f_k(y^k)\|\le\kappa_{\theta}\min(1,\|\nabla\!f(x^k)\|^{1/2})\|\nabla\!f(x^k)\|,
 \end{equation*}
 which is another one used by Cartis et al. \cite{Cartis11a,Cartis11b} for the proposed methods.
\end{remark}

 Note that $v^k=R_{k,j_k}(y^k)$ for each $k\in\mathbb{N}$. By the expression of $R_{k,j_k}$, we have 
 \begin{equation}\label{yk-vk}
  y^k\!-\!v^k={\rm prox}_{g}(y^k\!-\!\nabla\!f_{k,j_k}(y^k))\ \ {\rm and}\  \ y^k\!-\!\nabla\!f_{k,j_k}(y^k)\in\partial g(y^k\!-\!v^k)\quad\forall k\in\mathbb{N}.
 \end{equation}
 These two relations will be frequently used in the convergence analysis in Section \ref{sec4}. 

Before analyzing the properties of the iterate sequence yielded by Algorithm \ref{qregPNT}, we need to show that it is well-defined. This is implied by the following lemma.
\begin{lemma}\label{well-def}
 Fix any $k\in\mathbb{N}$ with $r(x^k)\ne 0$. The following assertions hold.
 \begin{itemize}
  \item[(i)] For each $j\in\mathbb{N}$, subproblem \eqref{subprobj} has a stationary point $\overline{y}^{k,j}$ satisfying $\Theta_{k,j}(\overline{y}^{k,j})<\Theta_{k,j}(x^k)$.
		
  \item[(ii)] For each $j\in\mathbb{N}$, condition \eqref{inexact-cond} is satisfied by any $y-R_{k,j}(y)$ with $y$ close enough to a stationary point $\overline{y}^{k,j}$ of \eqref{subprobj} with $\Theta_{k,j}(\overline{y}^{k,j})<\Theta_{k,j}(x^k)$. 
		
 \item[(iii)] The inner for-end loop stops within a finite number of steps.
 \end{itemize}
\end{lemma}
\begin{proof}
 {\bf (i)} As $g$ is convex and its domain is nonempty, we have ${\rm ri}({\rm dom}\,g)\ne\emptyset$. Pick any $\widehat{x}\in{\rm ri}({\rm dom}\,g)$. Since $\partial g(\widehat{x})\ne\emptyset$, we choose any $\widehat{\xi}\in\partial g(\widehat{x})$. Then, 
 \begin{equation}\label{g-ineq1}
  g(x)\ge g(\widehat{x})+\langle\widehat{\xi},x-\widehat{x}\rangle\quad\ \forall x\in{\rm dom}\,g,
 \end{equation}
 which means that $f_{k,j}(z)+g(z)\ge\vartheta_{k}(z)+\langle\widehat{\xi},z\rangle+(L_{k,j}/q)\|z-x^k\|^q+g(\widehat{x})-\langle\widehat{\xi},\widehat{x}\rangle$ for any $z\in\mathbb{X}$.  Note that $L_{k,0}>\max(0,-\lambda_{\min}(\nabla^2f(x^k)))$ when $q=2$. Together with $q\in[2,3]$ and the expression of $\vartheta_k$, it follows that for each $j\in\mathbb{N}$, as $\|z\|\to\infty$,
  \[
   \vartheta_{k}(z)+\langle\widehat{\xi},z\rangle+(L_{k,j}/q)\|z-x^k\|^q+g(\widehat{x})-\langle\widehat{\xi},\widehat{x}\rangle\to\infty.
 \]
 That is, for each $j\in\mathbb{N}$, subproblem \eqref{subprobj} is coercive, so has an optimal solution.
 Fix any $j\in\mathbb{N}$. Let $\overline{y}^{k,j}$ be an optimal solution of \eqref{subprobj}. Clearly, $\overline{y}^{k,j}$ is a stationary point of subproblem \eqref{subprobj} and $\Theta_{k,j}(\overline{y}^{k,j})\le\Theta_{k,j}(x^k)$. The rest only argues that $\Theta_{k,j}(\overline{y}^{k,j})\ne\Theta_{k,j}(x^k)$. If not, then $x^k$ is an optimal solution of subproblem \eqref{subprobj}, and $0\in\nabla\!f_{k,j}(x^k)+\partial g(x^k)=\nabla\!f(x^k)+\partial g(x^k)$, a contradiction to $r(x^k)\ne 0$. 
	
 \noindent
 {\bf(ii)} Fix any $j\in\mathbb{N}$.  Let $\overline{y}^{k,j}$ be a stationary point of \eqref{subprobj} with $\Theta_k(\overline{y}^{k,j})<\Theta_k(x^k)$. Clearly, $\overline{y}^{k,j}\neq x^k$. From $r_{k,j}(\overline{y}^{k,j})=0$, we have
 $r_{k,j}(\overline{y}^{k,j})-\varrho L_{k,j}\|\overline{y}^{k,j}\!-x^k\|^{q-1}<0$.
 Note that $r_{k,j}(\cdot)$ is continuous because ${\rm prox}_g$ is nonexpansive and $\nabla\!f_{k,j}$ is continuous. Then, for any $y$ sufficiently close to $\overline{y}^{k,j}$, $r_{k,j}(y)\!-\!\varrho L_{k,j}\|y\!-\!x^k\|^{q-1}\le0$. The rest only argues that for any $y$ close enough to $\overline{y}^{k,j}$, $\Theta_k(y\!-\!R_{k,j}(y))\le\Theta_k(x^k)$. Indeed, from the continuity of $R_{k,j}$ and $R_{k,j}(\overline{y}^{k,j})=0$, we have $\big(y\!-\!R_{k,j}(y),R_{k,j}(y)\!-\!\nabla\!f_{k,j}(y)\big)\to\big(\overline{y}^{k,j},-\nabla\!f_{k,j}(\overline{y}^{k,j})\big)$ as $y\to\overline{y}^{k,j}$.
 In addition, by the expression of $R_{k,j}$, for any $z\in\mathbb{X}$, $R_{k,j}(z)-\nabla\!f_{k,j}(z)\in\partial g(z\!-\!R_{k,j}(z))$. Then, the subdifferential continuity of $g$ at $\overline{y}^{k,j}$ by \cite[Example 13.30]{RW98} implies that $g(y\!-\!R_{k,j}(y))\to g(\overline{y}^{k,j})$ as $y\to\overline{y}^{k,j}$. Thus, $\Theta_{k,j}(y-\!R_{k,j}(y))\to\Theta_{k,j}(\overline{y}^{k,j})$ as $y\to\overline{y}^{k,j}$. Together with  $\Theta_{k,j}(\overline{y}^{k,j})<\Theta_{k,j}(x^k)$, we have $\Theta_{k,j}(y-R_{k,j}(y))\le\Theta_{k,j}(x^k)$. The above arguments show that for any $y$ close enough to $\overline{y}^{k,j}$, the vector $y-R_{k,j}(y)$ satisfies condition \eqref{inexact-cond}.  
	
 \noindent
 {\bf (iii)} Suppose on the contrary that the result does not hold. For all $j$ large enough,
 \begin{equation}\label{undescent}
  F(y^{k,j})>F(x^k)-\sigma(L_{k,j}/q)\|y^{k,j}-x^k\|^q.
 \end{equation}
 From $\Theta_{k,j}(y^{k,j})\le\Theta_{k,j}(x^k)$ and the expression of $\Theta_{k,j}$, for each $j$ large enough,
 \begin{equation}\label{ineq-inexactb}
  \vartheta_k(y^{k,j})+(L_{k,j}/q)\|y^{k,j}-x^k\|^q+g(y^{k,j})\le F(x^k).
 \end{equation}
 Invoking inequality \eqref{g-ineq1} with $x=y^{k,j}$ and inequality \eqref{ineq-inexactb} leads to  
 \begin{align*}
  F(x^k)&\ge\vartheta_k(y^{k,j})+(L_{k,j}/q)\|y^{k,j}\!-\!x^k\|^q
	+\langle\widehat{\xi},y^{k,j}\!-\!\widehat{x}\rangle+g(\widehat{x}),\\
  &\ge f(x^k)-\|\nabla\!f(x^k)\|\,\|y^{k,j}-x^k\|-\|\widehat{\xi}\|\|y^{k,j}\!-\!\widehat{x}\|+g(\widehat{x})\\
 &\quad+\frac{1}{2}\lambda_{\rm min}(\nabla^2\!f(x^k))\|y^{k,j}\!-\!x^k\|^2+(L_{k,j}/q)\|y^{k,j}\!-\!x^k\|^q,
 \end{align*}
 where the second inequality is due to the expression of $\vartheta_k(y^{k,j})$. Recall that $L_{k,j}\to\infty$ as $j\to\infty$. The above inequality, along with $q\in[2,3]$, implies that $\lim_{j\to\infty}y^{k,j}=x^k$. Next we claim that $\liminf_{j\to\infty}L_{k,j}\|y^{k,j}-x^k\|^{q-1}>0$. If not, there exists $J\subset\mathbb{N}$ such that $\liminf_{J\ni j\to\infty}L_{k,j}\|y^{k,j}-x^k\|^{q-1}=0$, which by the first inequality of \eqref{inexact-cond} implies that $\lim_{J\ni j\to\infty}r_{k,j}(y^{k,j})=0$. Along with the expression of $R_{k,j}$, we have
 \begin{equation*}
  0=\lim_{J\ni j\to\infty}\big[y^{k,j}\!-\!{\rm prox}_g(y^{k,j}\!-\!\nabla\!\vartheta_{k}(y^{k,j})\!-\!L_{k,j}\|y^{k,j}\!-\!x^k\|^{q-2}(y^{k,j}\!-\!x^k))\big],
 \end{equation*}
 which by $\lim_{j\to\infty}y^{k,j}=x^k$ implies that $r(x^k)=0$, a contradiction to $r(x^k)\ne 0$. Thus, $\liminf_{j\to\infty}L_{k,j}\|y^{k,j}\!-\!x^k\|^{q-1}>0$,
 which implies that for all sufficiently large $j\in\mathbb{N}$, 
 $(1\!-\!\sigma)({L_{k,j}}/{q})\|y^{k,j}-x^k\|^q\ge o(\|y^{k,j}-x^k\|^2)$. Together with $\lim_{j\to\infty}y^{k,j}=x^k$, the second-order Taylor expansion of $f$ at $x^k$ and $q\in[2,3]$, for all $j\in\mathbb{N}$ large enough, 
 \begin{align*}
  F(y^{k,j})-F(x^k)&=\vartheta_k(y^{k,j})+o(\|y^{k,j}-x^k\|^2)+g(y^{k,j})-F(x^k)\\
 &\le \vartheta_k(y^{k,j})+g(y^{k,j})-F(x^k)+(1\!-\!\sigma)({L_{k,j}}/{q})\|y^{k,j}\!-x^k\|^q.
 \end{align*}
 Combining this with \eqref{ineq-inexactb} yields that $F(y^{k,j})-F(x^k)\le-\sigma (L_{k,j}/q)\|y^{k,j}\!-x^k\|^q$, a contradiction to inequality \eqref{undescent}. The conclusion then follows. 
\end{proof}

 For convenience, in the rest of this paper, we let $\{x^k\}_{k\in\mathbb{N}}$ and $\{y^k\}_{k\in\mathbb{N}}$ be the sequences generated by Algorithm \ref{qregPNT}. Next we take a closer look at the  properties of the iterate and objective value sequences of Algorithm \ref{qregPNT}.  
\begin{proposition}\label{prop1-xk}  
 Under Assumption \ref{ass1}, the following assertions hold. 
 \begin{itemize}
 \item [(i)] For each $k\in\mathbb{N}$, $F(x^{k+1})\le F(y^k)\le F(x^k)-\sigma(L_{m}/q)\|y^k\!-\!x^k\|^q$; consequently, the sequence $\{F(x^k)\}_{k\in\mathbb{N}}$ is nonincreasing and convergent. 
		
 \item[(ii)] $\sum_{k=1}^{\infty}\|y^k\!-\!x^k\|^q<\infty$ and hence $\lim_{k\to\infty}(y^k-x^k)=0$.
		
\item [(iii)] If $\{x^k\}_{k\in\mathbb{N}}$ is bounded, then the sequence $\{y^k\}_{k\in\mathbb{N}}$ is bounded, and the set of accumulation points of $\{x^k\}_{k\in\mathbb{N}}$, denoted by $\omega(x^0)$, is nonempty and compact.
\end{itemize}
\end{proposition} 

For the subsequent convergence analysis, we need the following assumption:
\begin{aassumption}\label{ass2}
{\bf(i)} The sequence $\{x^k\}_{k\in\mathbb{N}}$ is bounded; 
	
 \noindent
 {\bf(ii)} $\nabla^2\!f$ is strictly continuous on an open convex set $\mathcal{N}$ with $\mathcal{O}\supset\mathcal{N}\!\supset\omega(x^0)$.
\end{aassumption}

 Assumption \ref{ass2} (i) is implied by the boundedness of the level set of $F$ associated with $F(x^0)$. Assumption \ref{ass2} (ii) is weaker than its global version, a common condition for the convergence rate analysis of the CR method (see \cite{Doikov20,Doikov22,Yue19cubic}). Under Assumptions \ref{ass1}-\ref{ass2}, we can establish the following conclusion for the sequences $\{x^k\}_{k\in\mathbb{N}}$ and $\{y^k\}_{k\in\mathbb{N}}$, which plays a crucial role in the convergence analysis of Section \ref{sec4}.
\begin{proposition}\label{prop2-xk}  
 Under Assumptions \ref{ass1}-\ref{ass2}, the following statements hold.  
 \begin{itemize}	
 \item[(i)] There exists a constant $\widehat{L}>0$ such that $L_k\le\widehat{L}$ for all $k\in\mathbb{N}$.
		
 \item[(ii)] $\lim_{k\to\infty}v^k=\lim_{k\to\infty}R_k(y^k)=0$ and $\lim_{k\to\infty}r(x^k)=0$.
		
 \item[(iii)] $\omega(x^0)$ is a nonempty connected compact set, and $\omega(x^0)\subset\mathcal{S}^*$. 
 
 \item[(iv)] For each $\overline{x}\in\omega(x^0)$, there exists a subsequence $\{x^{k_{\nu}}\}_{\nu\in\mathbb{N}}$ such that $x^{k_{\nu}}\to\overline{x}$ and $F(x^{k_{\nu}})\to F(\overline{x})$ as $\nu\to\infty$, and hence $F(\overline{x})=\lim_{k\to\infty}F(x^k)\!:=\overline{F}$.
\end{itemize}
\end{proposition}
\begin{proof}
 {\bf (i)} Suppose on the contrary that the result does not hold. There exists $\mathcal{K}\!\subset\mathbb{N}$ such that $\lim_{\mathcal{K}\ni k\to\infty}L_k\!=\infty$. For each $k\in\mathcal{K}$, let $\widetilde{L}_k\!:=\tau^{-1}L_k$. Clearly, $j_{k}>1$ for all sufficiently large $k\in\mathcal{K}$. From step (2b), for all sufficiently large $k\in\mathcal{K}$,
 \begin{equation}\label{ineq31}
  F(y^{k,j_k\!-\!1})>F(x^k)-\sigma(\widetilde{L}_k/q)\|y^{k,j_k\!-\!1}\!-\!x^k\|^q.
 \end{equation}
 Note that $\Theta_{k,j_{k}-1}(y^{k,j_{k}-1})\le\!\Theta_{k,j_{k}-1}(x^k)\!=F(x^k)$ and  $L_{k,j_k-1}\!=\widetilde{L}_k$. For each $k\in\mathcal{K}$, 
 \begin{align}\label{ineq32}
  q^{-1}\widetilde{L}_{k}\|y^{k,j_k\!-\!1}\!-\!x^k\|^q
  &\le F(x^k)\!-\!g(y^{k,j_k\!-\!1})-\vartheta_k(y^{k,j_k\!-\!1})\\
  &\le F(x^0)\!-\!g(y^{k,j_k\!-\!1})-\vartheta_k(y^{k,j_k\!-\!1})\nonumber\\
  &\le F(x^0)-g(\widehat{x})+\langle\widehat{\xi},\widehat{x}\!-\!y^{k,j_k\!-\!1}\rangle-\vartheta_k(y^{k,j_k\!-\!1})\nonumber\\
  &\le F(x^0)\!-\!f(x^k)-g(\widehat{x})+\|\widehat{\xi}\|\big(\|y^{k,j_k\!-\!1}\!-\!x^k\|\!+\!\|x^k\!-\!\widehat{x}\|\big)\nonumber\\
  &\quad+\|\nabla\!f(x^k)\|\|y^{k,j_k\!-\!1}\!-\!x^k\|\!+\!\frac{1}{2}\|\nabla^2\!f(x^k)\|\|y^{k,j_k\!-\!1}\!-\!x^k\|^2\nonumber
  \end{align}
  where the second inequality is by Proposition \ref{prop1-xk} (i), and the third one is obtained by using \eqref{g-ineq1} with $x=y^{k,j_k\!-\!1}$. From Assumption \ref{ass2} (i), the continuity of $\nabla^2\!f$ on ${\rm dom}\,g$ and $\lim_{\mathcal{K}\ni k\to\infty}\widetilde{L}_k=\infty$, the above inequality implies that $\lim_{\mathcal{K}\ni k\to\infty}\|y^{k,j_k\!-\!1}\!-\!x^k\|=0$. Then, for each $k\in\mathcal{K}$ large enough, using the mean-value theorem and Assumption \ref{ass2} (ii) and following the same arguments as those for \cite[Lemma 4.2 (iv)]{QianPan22} leads to $(1\!-\!\sigma)(\widetilde{L}_k/q)\|y^{k,j_k\!-\!1}\!-\!x^k\|^q\le\frac{c_0}{2}\|y^{k,j_k\!-\!1}\!-\!x^k\|^{3}$ for some $c_0>0$, which is impossible because $\lim_{\mathcal{K}\ni k\to\infty}\widetilde{L}_k=\infty$, $y^{k,j_k\!-\!1}\neq x^k$ by \eqref{ineq31} and $\lim_{\mathcal{K}\ni k\to\infty}\|y^{k,j_k\!-\!1}\!-\!x^k\|=0$. 
	
 \noindent
 {\bf(ii)} By Proposition \ref{prop1-xk} (ii), we have $\lim_{k\to\infty}\|y^k\!-\!x^k\|=0$. Along with part (i), the inexactness condition \eqref{inexact-cond} and the definition of $v^k$, we have $\lim_{k\to\infty}v^k=0$. By the definition of function $r$ in \eqref{Rmap} and the nonexpansiveness of ${\rm prox}_{g}$, for each $k\in\mathbb{N}$, 
 \begin{align*}
  &r(x^k)\le\|R(x^k)\!-\!R_k(y^k)\|+r_k(y^k)\nonumber\\
  &\le\|x^k\!-\!y^k+{\rm prox}_g(y^k\!-\!\nabla\!f_k(y^k)) 
     -{\rm prox}_g(x^k\!-\!\nabla\!f(x^k))\|+\varrho L_k\|y^k\!-\!x^k\|^{q-1}\nonumber\\
  &\le2\|y^k\!-\!x^k\|+\|\nabla^2\!f(x^k)(y^k\!-\!x^k)+L_k\|y^k\!-\!x^k\|^{q-2}(y^k\!-\!x^k)\|+\varrho L_k\|y^k\!-\!x^k\|^{q-1}\nonumber\\ 
  &\le2\|y^k\!-\!x^k\|+\|\nabla^2\!f(x^k)\|\|y^k\!-\!x^k\|+(\varrho\!+\!1)L_k\|y^k\!-\!x^k\|^{q-1}.
 \end{align*}
 Passing the limit $k\to\infty$ to this inequality and using Proposition \ref{prop1-xk} (ii), Assumption \ref{ass2} (i) and the boundeness of $\{L_k\}_{k\in\mathbb{N}}$ in part (i) yields that  $\lim_{k\to\infty}r(x^k)=0$.
	
 \noindent
 {\bf (iii)} By Proposition \ref{prop1-xk} (iii), it suffices to argue that $\omega(x^0)$ is connected. Note that $\lim_{k\to\infty}v^k=0$ by part (ii) and $\lim_{k\to\infty}(y^k-x^k)=0$. Along with \eqref{new-iter}, we have $\lim_{k\to\infty}\|x^{k+1}\!-\!x^k\|=0$. Using this limit and the same arguments as those for \cite[Lemma 5]{Bolte14} results in the connectedness of $\omega(x^0)$.
 Note that $r(\cdot)$ is continuous relative to ${\rm cl}\,({\rm dom}\,g)$ because ${\rm prox}_g$ is nonexpansive and $\nabla\!f$ is continuous relative to ${\rm cl}\,({\rm dom}\,g)$ by Assumption \ref{ass1} (i). The inclusion $\omega(x^0)\subset\mathcal{S}^*$ is immediate by the continuity of the function $r$ relative to ${\rm cl}\,({\rm dom}\,g)$ and part (ii).

 \noindent
 {\bf (iv)} Pick any $\overline{x}\in\omega(x^0)$. Then there exists a subsequence $\{x^{k_{\nu}}\}_{\nu\in\mathbb{N}}$ such that $x^{k_{\nu}}\to\overline{x}$ as $\nu\to\infty$. For each $k\in\mathbb{N}$, from the inclusion in \eqref{yk-vk} and the convexity of $g$,
 \begin{equation}\label{ineq-gconv}
 g(x)\ge g(y^k\!-\!v^k)+\langle v^k\!-\!\nabla\!f_{k,j_k}(y^k),x-(y^k\!-\!v^k)\rangle\quad\forall x\in{\rm dom}g.
 \end{equation} 
 From Assumption \ref{ass2} (i) and Proposition \ref{prop1-xk} (ii), the sequence $\{y^k\}_{k\in\mathbb{N}}$ is bounded and $\lim_{\nu\to\infty}y^{k_{\nu}}=\overline{x}$. Define $\mathcal{V}\!:=\{\nu\in\mathbb{N}\,|\,x^{k_{\nu}}=y^{k_{\nu}\!-\!1}\}$ and $\overline{\mathcal{V}}\!:=\mathbb{N}\backslash\mathcal{V}$. Obviously, $\overline{\mathcal{V}}=\{\nu\in\mathbb{N}\,|\,x^{k_{\nu}}=y^{k_{\nu}\!-\!1}\!-\!v^{k_{\nu}\!-\!1}\}$. For any $\nu\in\mathcal{V}$, by invoking \eqref{ineq-gconv} with $x=\overline{x}$ and $k=k_{\nu}\!-\!1$, from equation \eqref{new-iter} it follows that
 \begin{align*}
  &F(x^{k_{\nu}})=F(y^{k_{\nu}\!-\!1})<F(y^{k_{\nu}\!-\!1}\!-\!v^{k_{\nu}\!-\!1})\\
  &\le f(y^{k_{\nu}\!-\!1}\!-\!v^{k_{\nu}\!-\!1})+g(\overline{x})-\langle v^{k_{\nu}\!-\!1}\!-\!\nabla\!f_{k_{\nu}\!-\!1}(y^{k_{\nu}\!-\!1}),\overline{x}-(y^{k_{\nu}\!-\!1}\!-\!v^{k_{\nu}\!-\!1})\rangle\\
  &=f(x^{k_{\nu}}\!-\!v^{k_{\nu}\!-\!1})+g(\overline{x})-\langle v^{k_{\nu}\!-\!1}\!-\!\nabla\!f_{k_{\nu}\!-\!1}(x^{k_{\nu}}),\overline{x}-(x^{k_{\nu}}\!-\!v^{k_{\nu}\!-\!1})\rangle.
  \end{align*}
  Along with $\lim_{k\to\infty}v^k=0$ by part (ii) and the continuity of $\nabla\!f_{k_{\nu}-1}$ on ${\rm dom}\,g$, passing $\mathcal{V}\ni\nu\to\infty$ to the above inequality leads to $\limsup_{\mathcal{V}\ni \nu\to\infty}F(x^{k_{\nu}})\le F(\overline{x})$. Together with the lower semicontinuity of $F$, we get $\lim_{\mathcal{V}\ni \nu\to\infty}F(x^{k_{\nu}})=F(\overline{x})$. For any $\nu\in\overline{\mathcal{V}}$,  invoking \eqref{ineq-gconv} with $x=\overline{x}$ and $k=k_{\nu}-1$ and then passing the limit $\overline{\mathcal{V}}\ni\nu\to\infty$ to the obtained inequality leads to
  \begin{align*}
   \limsup_{\overline{\mathcal{V}}\ni\nu\to\infty}g(x^{k_{\nu}})
   &\le g(\overline{x})+\limsup_{\overline{\mathcal{V}}\ni\nu\to\infty}\,\langle v^{k_{\nu}\!-\!1}\!-\!\nabla\!f_{k_{\nu}\!-\!1}(y^{k_{\nu}\!-\!1}),(y^{k_{\nu}\!-\!1}\!-\!v^{k_{\nu}\!-\!1})-\overline{x}\rangle\\		&=g(\overline{x})+\limsup_{\overline{\mathcal{V}}\ni\nu\to\infty}\,\langle v^{k_{\nu}\!-\!1}\!-\!\nabla\!f_{k_{\nu}\!-\!1}(x^{k_{\nu}}+v^{k_{\nu}\!-\!1}),x^{k_{\nu}}-\overline{x}\rangle=g(\overline{x}),
   \end{align*}
   where the second equality is due to $x^{k_{\nu}}\to\overline{x}$ as $\nu\to\infty$. Together with the lower semicontinuity of $g$, we obtain  $\lim_{\overline{\mathcal{V}}\ni\nu\to\infty}g(x^{k_{\nu}})= g(\overline{x})$, so $\lim_{\overline{\mathcal{V}}\ni\nu\to\infty}F(x^{k_{\nu}})= F(\overline{x})$. Along with $\lim_{\mathcal{V}\ni\nu\to\infty}F(x^{k_{\nu}})=F(\overline{x})$, we obtain $\lim_{\nu\to\infty} F(x^{k_{\nu}})=F(\overline{x})$.
 \end{proof}
 \section{Convergence analysis of Algorithm \ref{qregPNT}}\label{sec4}

 We shall focus on the full convergence of $\{x^k\}_{k\in\mathbb{N}}$,  the local superlinear convergence rates of $\{x^k\}_{k\in\mathbb{N}}$ and $\{F(x^k)\}_{k\in\mathbb{N}}$,  and the global complexity bounds in terms of the first-order optimality conditions. First, we discuss the Lipschitz continuity of $\nabla\!f$ and $\nabla^2\!f$ on a compact convex set under Assumptions \ref{ass1}-\ref{ass2}. Recall that $\omega(x^0)$ is nonempty and compact by Proposition \ref{prop2-xk} (iii). From the openness and convexity of $\mathcal{N}\supset\omega(x^0)$, $\lim_{k\to\infty}{\rm dist}(x^k,\omega(x^0))=0$, $\lim_{k\to\infty}v^k=0$ and $\lim_{k\to\infty}\|x^k\!-\!y^k\|=0$, there necessarily exist a compact convex set $\Gamma\subset\mathcal{N}$, a constant $\overline{\varepsilon}>0$ and an index  $\overline{k}\in\mathbb{N}$ such that for all $k\ge \overline{k}$,
 \begin{align}\label{ineq-Gamma}
 x^k\in\!\!\bigcup_{x\in\omega(x^0)}\!\!\mathbb{B}(x,\overline{\varepsilon}/2)\subset\!\!\bigcup_{x\in\omega(x^0)}\!\!\mathbb{B}(x,\overline{\varepsilon})\subset\Gamma\ \ {\rm and}\ \ y^k,\,y^k\!-\!v^k\in\Gamma.
 \end{align}
 Recall that $\mathcal{N}\!\subset\mathcal{O}$. From $\Gamma\!\subset\mathcal{N}$ and Assumption \ref{ass1} (i), there exists $L_{g}\!>0$ such that
 \begin{equation}\label{grad-Lip}
  \|\nabla\!f(y)-\nabla\!f(x)\|\le L_{g}\|y-x\|\quad\forall x,y\in\Gamma.
 \end{equation}
 By Assumption \ref{ass2} (ii), there exists a constant $L_H>0$ such that for all $x,y\in\Gamma$
\begin{equation}\label{Hessf-Lip}
 \|\nabla^2\!f(x)-\nabla^2\!f(y)\|\le L_{H}\|x-y\|.
\end{equation}
 Along with the convexity of $\Gamma$ and the integral expression of $\nabla\!f(y)\!-\!\nabla\!f(x)$, we have
 \begin{equation}\label{Hessf-ineq}
  \|\nabla\!f(y)\!-\!\nabla\!f(x)\!-\!\nabla^2\!f(x)(y\!-\!x)\|
  \le ({L_H}/{2})\|y-x\|^2\quad\ \forall x,y\in\Gamma.
 \end{equation}
 
\subsection{Global convergence}\label{sec4.1}

To establish the full convergence of the sequence $\{x^k\}_{k\in\mathbb{N}}$, 
 we need to introduce the following potential function
\begin{equation}\label{Phi-fun}
 \Phi(w)\!:=F(y-v)\quad\ \forall\, w=(y,v)\in\mathbb{X}\times\mathbb{X}. 
\end{equation}
 By \cite[Exercise 10.7]{RW98}, $\partial\Phi(w)=\{(\xi;-\xi)\ |\ \xi\in\partial F(y-v)\}$ for any $w=(y,v)$. Hence, if $w^*=(y^*,v^*)$ is a critical point of $\Phi$, then $x^*=y^*-v^*$ is a critical point of $F$; and conversely, if $x^*$ is a critical point of $F$, then $w^*=(x^*,0)$ is a critical point of $\Phi$. By the expression of $\Phi$ and the proof of \cite[Theorem 3.2]{LiPong18}, the following fact holds. 
 \begin{fact}\label{fact1}
  If $F$ is a KL function (of exponent $\theta\in[0,1)$), then so is the function $\Phi$.
\end{fact}
  
 In the rest of this section, for each $k\in\mathbb{N}$, write $w^k\!:=(y^k,v^k)$. By invoking condition \eqref{inexact-cond}, we can prove that ${\rm dist}(0,\partial\Phi(w^k))$ is upper bounded by $\|y^k\!-\!x^k\|^{q-1}$. 
\begin{lemma}\label{lem-relgap}
 Under Assumptions \ref{ass1}-\ref{ass2}, we have ${\rm dist}(0,\partial\Phi(w^k))\!\le\!\beta\|y^k\!-\!x^k\|^{q-1}$ for all $k\ge \overline{k}$, where $\beta:=\!\sqrt{2}[\varrho(1\!+\!L_g)\!+\!1]\widehat{L}+\!\frac{\sqrt{2}}{2}L_H$ with $\widehat{L}$ from Proposition \ref{prop2-xk} (i).
\end{lemma}
\begin{proof}
 By using previous \eqref{ineq-Gamma} and \eqref{Hessf-ineq} with $y=y^k$ and $x=x^k$, for all $k\ge \overline{k}$, 
 \begin{equation}\label{ineq-quad}
 \|\nabla\!f(y^k)-\nabla\!f(x^k)-\nabla^2\!f(x^k)(y^k\!-\!x^k)\|\le({L_H}/{2})\|y^k\!-\!x^k\|^{2}.
 \end{equation}  
 Fix any $k\ge\overline{k}$. From the inclusion in \eqref{yk-vk} and the expression of $\partial\Phi$, it holds that 
 \[
   \zeta^k\!:=\left(\begin{matrix}
   \nabla\!f(y^k\!-\!v^k)\!+\!v^k\!-\!\nabla\!f_{k,j_k}(y^k)\\
    -\nabla\!f(y^k\!-\!v^k)\!-\!v^k\!+\!\nabla\!f_{k,j_k}(y^k)
   \end{matrix}\right)\in\partial\Phi(w^k).
 \]
 By using the expression of $f_{k,j_k}$ and \eqref{grad-Lip} with $(y,x)=(y^k\!-\!v^k,y^k)\in\Gamma\!\times\Gamma$ leads to
 \begin{align}\label{zetak-bound}
  \|\zeta^k\|
  &=\sqrt{2}\big\| \nabla\!f(y^k\!-\!v^k)\!+\!v^k\!-\!\nabla\!f_{k,j_k}(y^k)\big\|\nonumber\\
  &=\sqrt{2}\|\nabla\!f(y^k\!-\!v^k)\!+\!v^k\!-\!\nabla\!f(x^k)\!-\!\nabla^2\!f(x^k)(y^k\!-\!x^k)\!-\!L_k\|y^k\!-\!x^k\|^{q-2}(y^k\!-\!x^k)\|\nonumber\\
  &\le\sqrt{2}\big[\|v^k\|+\|\nabla\!f(y^k\!-\!v^k)-\nabla\!f(y^k)\|+L_k\|y^k\!-\!x^k\|^{q-1}\nonumber\\
  &\qquad\quad +\|\nabla\!f(y^k)-\nabla\!f(x^k)-\nabla^2\!f(x^k)(y^k\!-\!x^k)\|\big]\nonumber\\
  &\stackrel{\eqref{ineq-quad}}{\le}\sqrt{2}\big[(1\!+\!L_g)\|v^k\|+({L_H}/{2})\|y^k\!-\!x^k\|^{2}+L_k\|y^k\!-\!x^k\|^{q-1}\big].
  \end{align}
  Note that $\|v^k\|=r_k(y^k)\le\varrho L_k\|y^k\!-\!x^k\|^{q-1}\le\varrho\widehat{L}\|y^k\!-\!x^k\|^{q-1}$ by condition \eqref{inexact-cond} and Proposition \ref{prop2-xk} (i); and $\|y^k\!-\!x^k\|^{2}\le\|y^k\!-\!x^k\|^{q-1}$ because $\|y^k\!-\!x^k\|<1$ for all $k\ge\overline{k}$ (if necessary by increasing $\overline{k}$) by Proposition \ref{prop1-xk} (ii). The conclusion then holds. 
\end{proof}

Lemma \ref{lem-relgap} provides a relative inexact optimality condition for minimizing $\Phi$ with the sequence $\{(x^k,y^k,v^k)\}_{k\in\mathbb{N}}$.  Unfortunately, the sufficient decrease of the sequence $\{\Phi(w^k)\}_{k\in\mathbb{N}}$ cannot be achieved though Proposition \ref{prop1-xk} (i) provides the sufficient decrease of $\{F(x^k)\}_{k\in\mathbb{N}}$ in terms of $\|y^k\!-\!x^k\|^{q-1}$. Thus, the recipe developed in \cite{Attouch13,Bolte14} and its generalized version in \cite{QianPan22} cannot be directly applied to obtain the convergence of the sequence $\{x^k\}_{k\in\mathbb{N}}$. Next we prove the full convergence of $\{x^k\}_{k\in\mathbb{N}}$ by skillfully combining the sufficient decrease of $\{F(x^k)\}_{k\in\mathbb{N}}$ with the KL inequality on $\Phi$.
\begin{theorem}\label{globalconv}
 Suppose that Assumptions \ref{ass1}-\ref{ass2} hold, and that $F$ is a KL function. Then, $\sum_{k=0}^{\infty}\|x^{k+1}\!-\!x^k\|<\infty$, and consequently, $\{x^k\}_{k\in\mathbb{N}}$ converges to a point $\overline{x}\in\mathcal{S}^*$. 
\end{theorem}
\begin{proof}
 If there exists $k_0\in\mathbb{N}$ such that $F(x^{k_0})=F(x^{{k_0}+1})$, by Proposition \ref{prop1-xk} (i) we have $y^{k_0}=x^{k_0}$, which along with condition \eqref{inexact-cond} means that $r(x^{k_0})=r_{k_0}(y^{k_0})=0$, so $x^{k_0}$ is a stationary point of \eqref{prob} and Algorithm \ref{qregPNT} stops within a finite number of steps. Hence, it suffices to consider that $F(x^k)>F(x^{k+1})$ for all $k\in\mathbb{N}$. 
 Note that the sequence $\{w^k\}_{k\in\mathbb{N}}$ is bounded. Denote by $W^*$ the set of its accumulation points. Recall that $\lim_{k\to\infty}\|y^k\!-\!x^k\|=0$ and $\lim_{k\to\infty}\|v^k\|=0$. It is easy to prove that $W^*=\{(y,0)\in\mathbb{X}\times\mathbb{X}\ |\ y\in\omega(x^0)\}$, which is nonempty and compact. Also, from Proposition \ref{prop2-xk} (iv), we have $\Phi(w)=\overline{F}$ for all $w\in W^*$. Note that $\Phi$ is a KL function by Fact \ref{fact1}. From \cite[Lemma 6]{Bolte14}, there exist $\varepsilon>0,\varpi>0$ and $\varphi\in\Upsilon_{\!\varpi}$ such that for all $w\in[\overline{F}<\Phi<\overline{F}+\varpi]\cap\mathfrak{B}(W^*,\varepsilon)$ with $\mathfrak{B}(W^*,\varepsilon)\!:=\big\{w\in\mathbb{X}\times\mathbb{X}\ |\ {\rm dist}(w,W^*)\le\varepsilon\big\}$, 
 \[
  \varphi'(\Phi(w)\!-\!\overline{F}){\rm dist}(0,\partial\Phi(w))\ge 1.
 \]
 For each $k\in\mathbb{N}$, from the expression of $\Phi$ and equation \eqref{new-iter}, it follows that
 \begin{equation}\label{relationPhi}
  \Phi(w^k)-\overline{F}=F(y^k\!-\!v^k)-\overline{F}\ge F(x^{k+1})-\overline{F}>0.
  \end{equation} 
  Next we claim that $\lim_{k\to\infty}\Phi(w^k)=\overline{F}$. Indeed, invoking \eqref{ineq-gconv} with $x=\overline{x}$ for some $\overline{x}\in\omega(x^0)$ yields that
  \(
   F(y^k\!-\!v^k)\le f(y^k\!-\!v^k)+g(\overline{x})-\langle v^k\!-\!\nabla\!f_{k,j_k}(y^k),\overline{x}-(y^k\!-\!v^k)\rangle,
  \)
 which by the boundedness of $\{(y^k,v^k)\}_{k\in\mathbb{N}}$ implies that $\{F(y^k\!-\!v^k)\}_{k\in\mathbb{N}}$ is bounded. Let $\mathcal{K}\subset\mathbb{N}$ be an index set such that $\lim_{\mathcal{K}\ni k\to\infty}F(y^k\!-\!v^k)\!=\limsup_{k\to\infty}F(y^k\!-\!v^k)$. By the boundedness of $\{y^k\!-\!v^k\}_{k\in\mathbb{N}}$, there is an index set $\mathcal{K}_1\subset\mathcal{K}$ such that $\{y^k\!-\!v^k\}_{k\in \mathcal{K}_1}$ is convergent with limit, denoted by $\widehat{x}$. Together with $\lim_{k\to\infty}(y^k\!-\!x^k)=0$ and $\lim_{k\to\infty}v^k=0$, we have $\widehat{x}\in\omega(x^0)$. By invoking \eqref{ineq-gconv} with $x=\widehat{x}$, for all $k\in\mathbb{N}$,
 \[
   F(y^k\!-\!v^k)\le f(y^k\!-\!v^k)+g(\widehat{x})-\langle v^k\!-\!\nabla\!f_{k,j_k}(y^k),\widehat{x}-(y^k\!-\!v^k)\rangle.
 \]
 Passing the limit $\mathcal{K}_1\ni k\to\infty$ to this inequality and using $\lim_{\mathcal{K}\ni k\to\infty}F(y^k\!-\!v^k)\!=\limsup_{k\to\infty}F(y^k\!-\!v^k)$ leads to $\limsup_{k\to\infty}F(y^k\!-\!v^k)\le
  F(\widehat{x})=\overline{F}$. Together with $\liminf_{k\to\infty}F(y^k\!-\!v^k)\ge\overline{F}$ implied by \eqref{relationPhi}, we have $\lim_{k\to\infty}\Phi(w^k)=\lim_{k\to\infty}F(y^k-v^k)=\overline{F}$. The claimed limit holds. Note that $\lim_{k\to\infty}{\rm dist}(w^k,W^*)=0$. For all $k\ge \overline{k}$ (if necessary by increasing $\overline{k}$), we have $w^k\in[\overline{F}<\Phi<\overline{F}+\varpi]\cap\mathfrak{B}(W^*,\varepsilon)$, and consequently, $\varphi'\big(\Phi(w^k)\!-\!\overline{F}\big){\rm dist}(0,\partial\Phi(w^k))\geq 1$. In addition, by Lemma \ref{lem-relgap}, for each $k\ge \overline{k}$, ${\rm dist}(0,\partial\Phi(w^k))\le \beta\|y^k\!-\!x^k\|^{q-1}$. Consequently, for each $k\ge\widehat{k}:=\overline{k}\!+\!1$, 
 \begin{equation}\label{ineq_KL}
  \beta\varphi'\big(\Phi(w^{k-1})-\overline{F}\big)\|y^{k-1}\!-\!x^{k-1}\|^{q-1}\ge 1.
  \end{equation}
  Fix any $k\ge\widehat{k}$. As $\varphi'$ is nonincreasing on $(0,\varpi)$, combining \eqref{relationPhi} and \eqref{ineq_KL} leads to
 \begin{equation}\label{equa1-later}
 \varphi'(F(x^k)\!-\!\overline{F})\ge\varphi'(\Phi(w^{k-1})\!-\!\overline{F})
  \ge\frac{1}{\beta\|y^{k-1}\!-\!x^{k-1}\|^{q-1}}.
 \end{equation}
 Together with the concavity of $\varphi$ and Proposition \ref{prop1-xk} (i), it follows that
 \begin{align*}
 \Delta_{k,k+1}
  &:=\varphi(F(x^k)\!-\!\overline{F})-\varphi(F(x^{k+1})\!-\!\overline{F})
     \ge\varphi'(F(x^k)-\overline{F})(F(x^{k})\!-\!F(x^{k+1}))\\
  &\ge\frac{F(x^{k})\!-\!F(x^{k+1})}{\beta\|y^{k-1}-x^{k-1}\|^{q-1}}
    \ge\frac{\sigma L_{m}\|y^k\!-\!x^k\|^q}{\beta q\|y^{k-1}\!-\!x^{k-1}\|^{q-1}},
 \end{align*}
 which can be rearranged as $\|y^k\!-\!x^k\|\le\|y^{k-1}\!-\!x^{k-1}\|^{\frac{q-1}{q}}\big[\beta q(\sigma L_{m})^{-1}\Delta_{k,k+1}\big]^{\frac{1}{q}}$. 
 Using the Young's inequality $u^{\frac{q-1}{q}}v^{\frac{1}{q}}\le\frac{q-1}{q}u+\frac{1}{q}v$ for $u\ge0, v\ge0$ yields that
 \begin{equation*}
  \|y^k\!-\!x^k\|
   \le q^{-1}(q-1)\|y^{k-1}\!-\!x^{k-1}\|+\beta(\sigma L_{m})^{-1}\Delta_{k,k+1}.
  \end{equation*}
 Summing this inequality from $\widehat{k}$ to any $l>\widehat{k}$ and using the nonnegativity of $\varphi$ leads to $q{\textstyle\sum_{i=\widehat{k}}^l}\|y^i-x^i\|
\le(q\!-\!1)\big[{\textstyle\sum_{i=\widehat{k}}^l}\|y^i\!-\!x^i\|
		+\|y^{\widehat{k}-1}\!-\!x^{\widehat{k}-1}\|\big]
		+\frac{\beta q\varphi(F(x^{\widehat{k}})-\overline{F})}{\sigma L_{m}}$. Then,
 \begin{equation*}
  {\textstyle\sum_{i=\widehat{k}}^l}\|y^i-x^i\|\le(q\!-\!1)\|y^{\widehat{k}-1}\!-\!x^{\widehat{k}-1}\|
		+\beta q(\sigma L_{m})^{-1}\varphi(F(x^{\widehat{k}})-\overline{F}).
 \end{equation*}
 Passing the limit $l\to\infty$ to this inequality results in $\sum_{i=\widehat{k}}^{\infty}\|y^i-x^i\|<\infty$. From $q-1\ge 1$, it follows that  $\sum_{i=\widehat{k}}^{\infty}\|y^i-x^i\|^{q-1}<\infty$. 
 Recall that $\|v^k\|\le \varrho L_k\|y^k\!-\!x^k\|^{q-1}$ and $L_k\le\widehat{L}$ for all $k\in\mathbb{N}$. Hence, 
 ${\textstyle\sum_{i\ge\widehat{k}}}\|v^i\|<\infty$. Now, from \eqref{new-iter}, we have
 \begin{align*}
  {\textstyle\sum_{i\ge \widehat{k}}}\|x^{i+1}-x^i\|
  &\le{\textstyle\sum_{i\ge\widehat{k}}}\|y^i-x^i\|+{\textstyle\sum_{i\ge\widehat{k}}}\|v^i\|<\infty.
  \end{align*}
 This shows that the desired result holds. The proof is completed. 
\end{proof}

\subsection{Local convergence rates}\label{sec4.2}

 By Fact \ref{fact1}, if $F$ is a KL function of exponent $\theta\in[0,1)$, then $\Phi$ is a KL function of exponent $\theta\in[0,1)$. Based on the proof of Theorem \ref{globalconv}, by using  Proposition \ref{prop1-xk} and following the similar arguments to those of \cite[Theorem 3.2]{QianPan22}, one can achieve the following local convergence rate result. 
\begin{theorem}\label{KL-rate}
 Suppose that Assumptions \ref{ass1}-\ref{ass2} hold, and that $F$ is a KL function of exponent $\theta\in(0,1)$. Then, the sequence $\{x^k\}_{k\in\mathbb{N}}$ converges to a point $\overline{x}\in\mathcal{S}^*$, and 
 \begin{itemize}
 \item[(i)] if $\theta\in\!(0,\frac{q-1}{q})$, for all sufficiently large $k$, $\|x^{k+1}-\overline{x}\|\le\varepsilon\|x^{k}-\overline{x}\|^{\frac{q-1}{\theta q}}$ for any given $\varepsilon\in(0,1)$, and $F(x^{k+1})-\overline{F}\le\alpha_0(F(x^k)-\overline{F})^{\frac{q-1}{\theta q}}$ for some $\alpha_0>0$; 
 
 \item[(ii)] if $\theta\!\in\![\frac{q-1}{q},1)$, there are $\alpha\!>0$ and $\varrho\in\!(0,1)$ such that for all $k$ large enough,
		\begin{equation*}
		   \|x^k-\overline{x}\|
			\le\left\{\begin{array}{cl}
				\alpha\varrho^{k} &{\rm if}\ \theta=\frac{q-1}{q},\\
				\!\alpha k^{\frac{1-\theta}{1-(q\theta)/(q-1)}}&{\rm if}\ \theta\in(\frac{q-1}{q},1).
			\end{array}\right.
		\end{equation*}
	\end{itemize}
\end{theorem}

Theorem \ref{KL-rate} (i) shows that when $F$ is a KL function of exponent $\theta\in(0,\frac{q-1}{q})$, $\{x^k\}_{k\in\mathbb{N}}$ and $\{F(x^k)\}_{k\in\mathbb{N}}$ converge superlinearly to $\overline{x}\in\mathcal{S}^*$ and $\overline{F}$, respectively, with order $\frac{q-1}{\theta q}$, which is specified as $4/3$ for $q=3$ and $\theta=1/2$. Also, this order is the best under the KL property of $F$ with exponent $1/2$. In the rest of this section, we improve this superlinear convergence rate result for $q>2$ under a local H\"{o}lderian error bound on the set $\mathcal{X}^*$. To attain this goal, we need the following lemma. 
\begin{lemma}\label{lemma-relation}
 Fix any $x^*\!\in\!\mathcal{X}^*\cap\Gamma$. For each $k\ge\overline{k}$, the following inequalities hold:
 \begin{align*}
  L_k\|y^k\!-\!x^k\|^{q-2}\|y^k\!-\!v^k\!-\!x^{*}\|^2
  \le\|y^k\!-\!v^k\!-\!x^{*}\|\big[(L_H/2)\|x^k\!-\!x^{*}\|^{2}\qquad\qquad\ \\
  \quad +(L_k\|y^k\!-\!x^k\|^{q-3}\!+\!L_H)\|x^{k}\!-\!x^{*}\|\|y^k\!-\!x^k\|\\
  \quad +\big(1\!+\!L_k\|y^k\!-\!x^k\|^{q-2}\!+\!\|\nabla^2\!f(x^{*})\|\big)\|v^k\|\big],\\
   F(y^k\!-\!v^k)-F(x^*)\le\frac{2L_H}{3}\|y^k\!-\!v^k\!-\!x^{*}\|^3+\frac{\beta}{\sqrt{2}}\|y^k\!-\!x^k\|^{q-1}\|y^k\!-\!v^k\!-\!x^{*}\|.
 \end{align*}
 where the constant $\beta$ is the same as the one in Lemma \ref{lem-relgap}. 
 \end{lemma}
\begin{proof}
 Fix any $k\ge\overline{k}$. Let $z^k:=y^k\!-\!v^k$. From the equality in \eqref{yk-vk}, we have $z^k={\rm prox}_{g}(y^k\!-\!\nabla\!f_{k,j_k}(y^k))$, 
 which by $x^{*}\in\mathcal{X}^*$ and the definition of $\mathcal{X}^*$ means that 
 \begin{equation}\label{ineq-spoint}
  \langle z^k-x^*,\nabla^2\!f(x^*)(z^k\!-\!x^*)\rangle\ge 0.
 \end{equation}
 In addition, combining the inclusion in \eqref{yk-vk} and the expression of $f_{k,j_k}$ leads to 
 \begin{equation*}
  v^{k}\!-\!\nabla\!f_{k,j_k}(y^k)\!=v^{k}-\!\nabla\!f(x^k)-\!\nabla^2\!f(x^k)(y^k\!-\!x^k)\!
  -\!L_k\|y^k\!-\!x^k\|^{q-2}(y^k\!-\!x^k)\!\in\partial g(z^k).
 \end{equation*}
 Together with $0\in\nabla\!f(x^{*})+\partial g(x^{*})$ and the monotonicity of $\partial g$, it follows that
 \begin{align*}
  0&\le\langle z^k-x^{*},\nabla\!f(x^{*})+v^k\!-\!\nabla\!f(x^k)-\nabla^2\!f(x^k)(y^k\!-\!x^k)\rangle\\
   &\quad-L_k\|y^k\!-\!x^k\|^{q-2}\langle z^k-x^{*},y^k\!-\!x^k\rangle,
  \end{align*}
 which, after a suitable rearrangement, can be equivalently written as
  \begin{align*}
  &\langle z^k\!-\!x^{*},\big[L_k\|y^k\!-\!x^k\|^{q-2}\mathcal{I}+\nabla^2\!f(x^{*})\big](y^k\!-\!x^{*})\rangle\\
  &\le\langle z^k\!-\!x^{*},\nabla\!f(x^{*})\!-\!\nabla\!f(x^k)\!-\!\nabla^2\!f(x^{*})(x^{*}-x^k)\rangle+\langle z^k\!-\!x^{*},v^k\rangle\\
  &\quad-\langle z^k\!-\!x^{*},L_k\|y^k\!-\!x^k\|^{q-2}(x^{*}\!-\!x^k)
		+\big[\nabla^2\!f(x^k)\!-\!\nabla^2\!f(x^{*})\big](y^k\!-\!x^k)\rangle.
  \end{align*}
   Adding the term $\langle x^{*}\!-\!z^k,\big[L_k\|y^k\!-\!x^k\|^{q-2}\mathcal{I}	+\nabla^2\!f(x^{*})\big]v^k\rangle$ to this inequality leads to
  \begin{align*}
   &\langle z^k\!-\!x^{*},\big[L_k\|y^k\!-\!x^k\|^{q-2}\mathcal{I}		+\nabla^2\!f(x^{*})\big](z^k\!-\!x^{*})\rangle\\
   &\le\langle z^k\!-\!x^{*},\nabla\!f(x^{*})-\nabla\!f(x^k)-\nabla^2\!f(x^{*})(x^{*}-x^k)\rangle+\langle z^k\!-\!x^{*},v^k\rangle\nonumber\\
   &\quad-\langle z^k\!-\!x^{*},L_k\|y^k\!-\!x^k\|^{q-2}(x^{*}\!-\!x^k)
  +\big[\nabla^2\!f(x^k)\!-\!\nabla^2\!f(x^{*})\big](y^k\!-\!x^k)\rangle\nonumber\\
  &\quad-\langle z^k\!-\!x^{*},\big[L_k\|y^k\!-\!x^k\|^{q-2}\mathcal{I}		+\nabla^2\!f(x^{*})\big]v^k\rangle.
  \end{align*}
 Together with the above \eqref{ineq-spoint} and the triangle inequality, it follows that 
 \begin{align*}
  &L_k\|y^k\!-\!x^k\|^{q-2}\|z^k\!-\!x^{*}\|^2\\
  &\le\|z^k\!-\!x^{*}\|\big[\|\nabla\!f(x^{*})-\nabla\!f(x^k)-\nabla^2\!f(x^{*})(x^{*}\!-\!x^k)\|\\
  &\quad+L_k\|y^k\!-\!x^k\|^{q-2}\|x^{*}\!-\!x^k\|
		+\|\nabla^2\!f(x^k)\!-\!\nabla^2\!f(x^{*})\|\|y^k\!-\!x^k\|+\|v^k\|\\
  &\quad+L_k\|y^k\!-\!x^k\|^{q-2}\|v^k\|+\|\nabla^2\!f(x^{*})\|\|v^k\|\big].
 \end{align*}
 Recall that $x^k,x^{*}\in\Gamma$. Invoking \eqref{Hessf-Lip}-\eqref{Hessf-ineq} with $y=x^*,x=x^k$ results in 
  \begin{align*}
  &\|\nabla\!f(x^{*})\!-\!\nabla\!f(x^k)\!+\!\nabla^2\!f(x^{*})(x^{k}\!-\!x^{*})\|+\|\nabla^2\!f(x^k)\!-\!\nabla^2\!f(x^{*})\|\|y^k\!-\!x^k\|\nonumber\\
  &\le (L_H/2)\|x^k\!-\!x^{*}\|^{2}+L_H\|x^{k}\!-\!x^{*}\|\|y^k\!-\!x^k\|.
  \end{align*}
 From the above two inequalities, we obtain the first desired inequality. 
 For the second one, combining $v^{k}\!-\!\nabla\!f_{k,j_k}(y^k)\in\partial g(z^k)$ with the convexity of $g$ leads to 
 \begin{equation}\label{gfun-ineq}
   g(z^k)
  \le g(x^*)+\langle v^{k}\!-\!\nabla\!f_{k,j_k}(y^k),z^k\!-\!x^*\rangle.
 \end{equation}
 In addition, from Assumption \ref{ass1} (i) and the integral formula, it follows that
 \begin{align*}
 f(z^k)\!-\!f(x^*)&=\langle\nabla\!f(x^*),z^k\!-\!x^*\rangle+\Big\langle z^k\!-\!x^*,\int_{0}^1\!\!\int_{0}^t\nabla^2\!f(x^*\!+\!s(z^k\!-\!x^*))(z^k\!-\!x^*)dsdt\Big\rangle\\
    &\!=\!\Big\langle z^k\!-\!x^*,\!\int_{0}^1\!\!\int_{0}^t\!\nabla^2\!f(x^*\!+\!s(z^k\!-\!x^*))(z^k\!-\!x^*)dsdt\Big\rangle\\ 
    &\quad+\langle\nabla\!f(z^k),z^k\!-\!x^*\rangle-\Big\langle z^k\!-\!x^*,\int_0^1\nabla^2f(x^*\!+t(z^k\!-\!x^*))(z^k\!-\!x^*)dt\Big\rangle\\
    &\stackrel{\eqref{ineq-spoint}}{\le}\!\Big\langle z^k\!-\!x^*,\!\int_{0}^1\!\!\int_{0}^t\!\big[\nabla^2\!f(x^*\!+\!s(z^k\!-\!x^*))\!-\!\nabla^2\!f(x^*)\big](z^k\!-\!x^*)dsdt\Big\rangle.\\ 
    &\qquad +\Big\langle z^k\!-\!x^*,\!\int_0^1\!\!\big[\nabla^2\!f(x^*)\!-\!\nabla^2\!f(x^*\!+\!t(z^k\!-\!x^*))\big](z^k\!-\!x^*)dt\Big\rangle\\
    &\qquad+\langle\nabla\!f(z^k),z^k\!-\!x^*\rangle\\
    &\stackrel{\eqref{Hessf-Lip}}{\le}(2L_H/3)\|z^k-x^*\|^3+\langle\nabla\!f(z^k),z^k-x^*\rangle.
    \end{align*}
 Adding this inequality to the above \eqref{gfun-ineq} and using inequality \eqref{zetak-bound} leads to   
\begin{align*}
 F(z^k)-F(x^*)
  &\le({2L_H}/{3})\|z^k\!-\!x^{*}\|^3+\langle\nabla\!f(z^k)+v^{k}\!-\!\nabla f_{k,j_k}(y^k),z^k\!-\!x^{*}\rangle\nonumber\\
  &\stackrel{\eqref{zetak-bound} }{\le}
  \frac{2L_H}{3}\|z^k\!-\!x^{*}\|^3+\frac{\beta}{\sqrt{2}}\|y^k\!-\!x^k\|^{q-1}\|z^k\!-\!x^{*}\|.
 \end{align*} 
 The second desired inequality then follows. The proof is completed.
 \end{proof}

 Now we bound $\|y^k\!-\!x^k\|$ with ${\rm dist}(x^k,\mathcal{X}^*)$ by the first inequality of Lemma \ref{lemma-relation}.
\begin{proposition}\label{prop-dkbound1}
 Fix any $q\in\!(2,3]$ and $\overline{x}\in\!\mathcal{X}^*$. Let $0<\!\varrho<\!\frac{1-\varsigma}{1+\varsigma+L_g}$ for some $\varsigma\in(0,1)$, and $\eta\!:=\!\frac{2+L_HL_{m}^{-1}}{2\eta_0}\!+\!\frac{\sqrt{2L_H(\eta_0L_{m})^{-1}+\eta_0^{-2}(2+L_HL_{m}^{-1})^2}}{2}$ for $\eta_0\!=1-\!\varsigma\!-\!\varrho(1\!+\!\varsigma\!+\!L_g)$. Under Assumptions \ref{ass1}-\ref{ass2}, 
 there is $\widetilde{k}\in\!\mathbb{N}$ such that for all $k\ge\widetilde{k}$ with $x^k\!\in\mathbb{B}(\overline{x},{\overline{\varepsilon}}/{2})$,
 \begin{equation}\label{ykxk-bound}
   \|y^k-x^k\|\le \eta{\rm dist}(x^k,\mathcal{X}^*).
 \end{equation}
\end{proposition}
\begin{proof}
 First of all, we claim that there exists $\overline{k}_1\in\!\mathbb{N}$ such that for all $k\ge\overline{k}_1$, 
 \begin{equation}\label{ykxk-relation}
  (1\!-\!\varsigma)\|y^k-x^k\|\le\|y^k-v^k-x^k\|\le(1\!+\!\varsigma)\|y^k-x^k\|.
 \end{equation}
 Recall that $\lim_{k\to\infty}(y^k\!-\!x^k)=0$ and $\{L_k\}_{k\in\mathbb{N}}$ is bounded. As $q>2$, there exists $\overline{k}_1\in\mathbb{N}$ such that for all $k\ge\overline{k}_1$, $\varrho L_k\|y^k\!-\!x^k\|^{q-2}\le\varsigma$. Along with condition \eqref{inexact-cond}, we have $\|v^k\|\le\varsigma\|y^k\!-\!x^k\|$ for all $k\ge\overline{k}_1$, which implies the inequalities in  \eqref{ykxk-relation}.
  
  Let $\widetilde{k}:=\max\{\overline{k},\overline{k}_1\}$. Fix any $k\ge \widetilde{k}$ with $x^k\in\mathbb{B}(\overline{x},\overline{\varepsilon}/2)$. From \eqref{ineq-Gamma}, there exists $\overline{x}^*\in\omega(x^0)$ such that $x^k\in\mathbb{B}(\overline{x}^*,\overline{\varepsilon}/2)\subset\Gamma$. Pick any $x^{k,*}\in\Pi_{\mathcal{X}^*}(x^k)$. Since $\|x^k-\overline{x}\|\ge{\rm dist}(x^k,\mathcal{X}^*)\ge\|x^{k,*}-\overline{x}^*\|-\|x^k-\overline{x}^*\|$, we have $\|x^{k,*}\!-\!\overline{x}^*\|\le\|x^k\!-\!\overline{x}\|+\|x^k\!-\!\overline{x}^*\|\le\overline{\varepsilon}$, i.e., $x^{k,*}\in\mathbb{B}(\overline{x}^*,\overline{\varepsilon})\subset\Gamma$. Note that $\|y^k\!-\!v^k\!-\!x^{k,*}\|\neq0$ (if not, by \eqref{ykxk-relation}, the result holds for $\eta=\frac{1}{1-\varsigma}$). Using the first inequality of Lemma \ref{lemma-relation} with $x^*\!=x^{k,*}$ leads to 
 \begin{align*}
   &L_k\|y^k\!-\!x^k\|^{q-2}\|y^k\!-\!v^k\!-\!x^{k,*}\|\\
   &\le (L_H/2)\|x^k\!-\!x^{k,*}\|^{2}+(L_H\!+\!L_k\|y^k\!-\!x^k\|^{q-3})\|y^k\!-\!x^k\|\|x^{k,*}\!-\!x^k\|\\
   &\quad+\varrho L_k\big[1\!+\!L_k\|y^k\!-\!x^k\|^{q-2}\!+\!\|\nabla^2\!f(x^{k,*})\|\big]\|y^k\!-\!x^k\|^{q-1}
  \end{align*} 
  where $\|v^k\|\le \varrho L_k\|y^k\!-\!x^k\|^{q-1}$ is also used.
  While from the above \eqref{ykxk-relation}, it follows that $\|y^k\!-\!v^k-x^{k,*}\|\ge(1-\varsigma)\|y^k\!-\!x^k\|-\|x^k\!-\!x^{k,*}\|$, and consequently,  
  \[
   L_k\|y^k\!-\!x^k\|^{q-2}\|y^k\!-\!v^k\!-\!x^{k,*}\|\ge L_k(1\!-\!\varsigma)\|y^k\!-\!x^k\|^{q-1}\!-\!L_k\|y^k\!-\!x^k\|^{q-2}\|x^{k,*}\!-\!x^k\|.
  \]
  From the above two inequalities, it immediately follows that 
  \begin{align*}
   (1\!-\!\varsigma)L_k\|y^k\!-\!x^k\|^{q-1}
   &\!\le\!\frac{L_H}{2}\|x^{k,*}\!-\!x^k\|^{2}\!+\!\big(L_H\!+\!2L_k\|y^k\!-\!x^k\|^{q-3}\big)\|y^k\!-\!x^k\|\|x^{k,*}\!-\!x^k\|\\
   &\quad\!+\!\varrho L_k \big[1\!+\!L_k\|y^k\!-\!x^k\|^{q-2}\!+\!\|\nabla^2\!f(x^{k,*})\|\big]\|y^k\!-\!x^k\|^{q-1}.
  \end{align*}
  From $x^{k,*}\in\Gamma$ and \eqref{grad-Lip}, $\|\nabla^2\!f(x^{k,*})\|\le L_g$. Recall that $\{L_k\}_{k\in\mathbb{N}}$ is bounded and $\lim_{k\to\infty}\|y^k\!-\!x^k\|=0$. If necessary by increasing $\widetilde{k}$, we have $L_k\|y^k\!-\!x^k\|^{q-2}\le\varsigma$ and $1\!+\!L_k\|y^k\!-\!x^k\|^{q-2}\!+\!L_g\le1\!+\!\varsigma\!+\!L_g$. By the definition of $\eta_0$ and the above inequality,
  \begin{align}\label{temp-ineqeta}
   \eta_0L_k\|y^k\!-\!x^k\|^{q-1}
   &\le ({L_H}/{2})\|x^{k,*}\!-\!x^k\|^{2}+L_H\|y^k\!-\!x^k\|\|x^{k,*}\!-\!x^k\|\nonumber\\
   &\quad\ \!+\!2L_k\|y^k\!-\!x^k\|^{q-2}\|x^k-x^{k,*}\|.
  \end{align}
 Multiplying \eqref{temp-ineqeta} with $(\eta_0L_k)^{-1}\|y^k\!-\!x^k\|^{3-q}$ and using 
  $\|y^k\!-\!x^k\|^{4-q}\le\|y^k\!-\!x^k\|$ and $\|y^k\!-\!x^k\|^{3-q}\le1$ (implied by $q\in(2,3]$ and $\|y^k\!-\!x^k\|\le1$ as $k\ge\widetilde{k}\ge\overline{k}$), we have 
  \begin{align*}
   \|y^k\!-\!x^k\|^{2}
   &\le L_H(2\eta_0L_k)^{-1}\|y^k\!-\!x^k\|^{3-q}{\rm dist}(x^k,\mathcal{X}^*)^{2}\\
   &\quad\!+\!2\eta_0^{-1}\|y^k\!-\!x^k\|{\rm dist}(x^k,\mathcal{X}^*)
		\!+\!L_H(\eta_0L_k)^{-1}\|y^k\!-\!x^k\|^{4-q}{\rm dist}(x^k,\mathcal{X}^*)\\
   &\le L_H(2\eta_0L_{m})^{-1}{\rm dist}(x^k,\mathcal{X}^*)^{2}\!+\!
		\eta_0^{-1}(2\!+\!L_HL_{m}^{-1}){\rm dist}(x^k,\mathcal{X}^*)\|y^k\!-\!x^k\|,
 \end{align*}
 where the second inequality is due to $L_k\ge L_{m}$. The above 
 quadratic inequality on $\|y^k\!-\!x^k\|$ implies that $\|y^k\!-\!x^k\|\le \eta{\rm dist}(x^k,\mathcal{X}^*)$. The proof is completed.
\end{proof}

 Next we use the relation in \eqref{ykxk-bound} to establish the superlinear rate of the residual sequence $\{r(x^k)\}_{k\in\mathbb{N}}$ and the iterate sequence $\{x^k\}_{k\in\mathbb{N}}$. Such a relation was used in the local convergence rate analysis of inexact proximal Newton-type methods \cite{Yue19,Mordu23,LiuPanWY22}. It is worth emphasizing that our proof is not a direct extension of theirs because the direction $y^k-x^k$ in the aforementioned works is obtained by minimizing a strongly convex function. When $g\equiv 0$ and the subproblems are solved exactly, such a result was achieved in \cite[Lemma 1]{Yue19cubic}, but its analysis is inapplicable to our algorithm. 
\begin{theorem}\label{Theorem-Lconverge}
 Fix any $q\in(2,3]$ and $\overline{x}\in\!\omega(x^0)$. Let  $0<\varrho<\!\frac{1-\varsigma}{1+\varsigma+L_g}$ for some $\varsigma\in(0,1)$. Suppose that Assumptions \ref{ass1}-\ref{ass2} hold, and that there exist $\widehat{\varepsilon}>0$ and $\kappa>0$ such that for all $x\in\mathbb{B}(\overline{x},\widehat{\varepsilon})$, ${\rm dist}(x,\mathcal{X}^*)\le\kappa[r(x)]^{\gamma}$ with $\gamma\in(\frac{1}{q-1},1]$. 
  Then, with $\widehat{c}\!:=\![(3+\!L_g)\varrho \widehat{L}\!+\!L_H/2\!+\!\widehat{L}]\eta^{q-1}$ where $\eta$ is the same as before, for all $k$ large enough, 
 \begin{equation}\label{aim-rineq}
 r(x^{k+1})\le \kappa\widehat{c}[r(x^k)]^{\gamma(q-1)}\ \ {\rm and}\ \ 
 \|x^{k+1}\!-\!\overline{x}\|
 \le2\eta\kappa(1\!+\!\varsigma)\widehat{c}^{\gamma}\|x^k\!-\!\overline{x}\|^{\gamma(q-1)},
 \end{equation}
 that is, the sequences $\{r(x^k)\}_{k\in\mathbb{N}}$ and $\{x^k\}_{k\in\mathbb{N}}$ converge to $0$ and $\overline{x}\in\mathcal{X}^*$, respectively, with the $Q$-superlinear rate of order $\gamma(q\!-\!1)$.
\end{theorem}
\begin{proof}
 From $\overline{x}\in\omega(x^0)$, there exists $\mathcal{K}\subset\mathbb{N}$ such that $\lim_{\mathcal{K}\ni k\to\infty}x^k=\overline{x}$. Along with the given assumption, ${\rm dist}(x^k,\mathcal{X}^*)\le\kappa[r(x^k)]^{\gamma}$ for all $k\in\mathcal{K}$ large enough. Passing the limit $\mathcal{K}\ni k\to\infty$ and using Proposition \ref{prop2-xk} (ii) leads to $\overline{x}\in\mathcal{X}^*$. 
 For each $k\ge\widetilde{k}$ where $\widetilde{k}$ is the same as in Proposition \ref{prop-dkbound1}, since $x^k\!\in\Gamma$ and $\overline{x}\in\omega(x^0)\!\subset\Gamma$, invoking the nonexpansiveness of ${\rm prox}_g$ and equation \eqref{grad-Lip} results in 
 \begin{equation}\label{rxk-ineq41}
  r(x^k)=|r(x^k)-r(\overline{x})|\le 2\|x^k-\overline{x}\|+\|\nabla\!f(x^k)-\nabla\!f(\overline{x})\|\le (2\!+\!L_g)\|x^k\!-\!\overline{x}\|.
 \end{equation} 
 Let $\varepsilon_1\!:=\min\{\widehat{\varepsilon},\overline{\varepsilon}\}/2$. From \eqref{new-iter} and \eqref{ykxk-relation}, for all $k\ge \widetilde{k}$, it holds that 
 \begin{equation}\label{xkp1-ineq41}
 (1\!-\!\varsigma)\|y^k\!-\!x^k\|\le\|x^{k+1}\!-\!x^k\|\le(1\!+\!\varsigma)\|y^k\!-\!x^k\|.
 \end{equation}
 Along with $\lim_{k\to\infty}(y^k\!-\!x^k)=0$, we get $\lim_{k\to\infty}(x^{k+1}\!-\!x^k)=0$, so $\|x^{k+1}\!-\!x^k\|\le\varepsilon_1$ for all $k\ge\widetilde{k}$ (if necessary by increasing $\widetilde{k}$). Thus, for all $k\ge \widetilde{k}$ with $x^k\in\mathbb{B}(\overline{x},\varepsilon_1)$, $\|x^{k+1}-\overline{x}\|\le2\varepsilon_1\le\min\{\widehat{\varepsilon},\overline{\varepsilon}\}$. Fix any $k\ge \widetilde{k}$ with $x^k\in\mathbb{B}(\overline{x},\varepsilon_1)$. Then,
 \begin{align}\label{ineq0-rk}
  r(x^{k+1})
  &\stackrel{ \eqref{yk-vk}}{=}\|x^{k+1}-{\rm prox}_g(x^{k+1}\!-\!\nabla\!f(x^{k+1}))+v^k-y^k+{\rm prox}_g(y^k\!-\!\nabla\!f_{k,j_k}(y^k))\|\nonumber\\
  &\le \|x^{k+1}-y^k-\nabla\!f(x^{k+1})+\nabla\!f_{k,j_k}(y^k)\|+\|x^{k+1}-y^k\|+\|v^k\|\nonumber\\
  &\le \|\nabla\!f(x^{k+1})-\nabla\vartheta_{k}(y^k)\|+2\|x^{k+1}-y^k\|+\|v_k\|+L_k\|y^k-x^k\|^{q-1}\nonumber\\
  &\stackrel{ \eqref{new-iter}}{\le}3\|v^k\|+\|\nabla\!f(x^{k+1})-\nabla\vartheta_{k}(y^k)\|+L_k\|y^k-x^k\|^{q-1},
  \end{align}
  where the first inequality is using the nonexpansiveness of ${\rm prox}_g$, and the second one is by the expression of $f_{k,j_k}$ in \eqref{subprobk}. By invoking \eqref{grad-Lip} and \eqref{ineq-quad}, we have
  \[
   \|\nabla\!f(x^{k+1})\!-\!\nabla\vartheta_k(y^k)\|
   \le L_g\|v^k\|+({L_H}/{2})\|y^k\!-\!x^k\|^{2}.
  \]
  Combining this inequality with \eqref{ineq0-rk} and using $\|v^k\|\le\varrho L_k\|y^k\!-\!x^k\|^{q-1}$ leads to  
  \begin{align}\label{rk-rate}
   r(x^{k+1})
   &\le (3+\!L_g)\varrho L_k\|y^k\!-\!x^k\|^{q-1}+({L_H}/{2})\|y^k\!-\!x^k\|^{2}+L_k\|y^k\!-\!x^k\|^{q-1}\nonumber\\
   &\le(3+\!L_g)\varrho \widehat{L}\|y^k\!-\!x^k\|^{q-1}+({L_H}/{2})\|y^k\!-\!x^k\|^{q-1}+\widehat{L}\|y^k\!-\!x^k\|^{q-1}\nonumber\\
   &\stackrel{\eqref{ykxk-bound}}{\le}[(3+\!L_g)\varrho \widehat{L}+L_H/2+\widehat{L}]\eta^{q-1}{\rm dist}(x^k,\mathcal{X}^*)^{q-1}\nonumber\\
   &\le[(3+\!L_g)\varrho \widehat{L}+L_H/2+\widehat{L}]\eta^{q-1}\kappa[r(x^k)]^{\gamma(q-1)},
  \end{align}
  where the second inequality is due to $\|y^k\!-\!x^k\|\le 1$ and the boundedness of $L_k$,  and the fourth one is by the given assumption. Fix  any $\sigma\in(0,1)$. Since $\lim_{k\to\infty}r(x^k)=0$ and $\gamma(q\!-\!1)>1$, from the above \eqref{rk-rate}, there exists $\varepsilon_2\in(0,\varepsilon_1)$ such that 
  \begin{equation}\label{rk-recursion}
   r(x^{k+1})\le\sigma r(x^k)\quad{\rm for\ all}\ k\ge \widetilde{k}\ {\rm with}\ x^k\in\mathbb{B}(\overline{x},\varepsilon_2).
  \end{equation}
  Let $\widetilde{\varepsilon}\!:=\min\{\frac{\varepsilon_2}{(1+\varsigma)\eta+1},\big(\frac{(1-\widetilde{\sigma})\varepsilon_2}{2(1+\varsigma)\eta\kappa(2+L_g)^{\gamma}}\big)^{1/\gamma},\frac{\varepsilon_2}{2}\}$ and $\widetilde{\sigma}:=\sigma^{\gamma}\in(0,1)$. 
  Recall that $\overline{x}\in\omega(x^0)$, there exists $\overline{k}_2\ge \widetilde{k}$ such that $x^{\overline{k}_2}\in\mathbb{B}(\overline{x},\widetilde{\varepsilon})$. We argue by induction that 
  \begin{equation}\label{induction}
  x^{k+1}\in\mathbb{B}(\overline{x},\varepsilon_2)\quad{\rm for\ all}\  k\ge \overline{k}_2. 
  \end{equation}
  Indeed, when $k=\overline{k}_2$, the conclusion holds by invoking $x^{\overline{k}_2}\in\mathbb{B}(\overline{x},\widetilde{\varepsilon})$ and noting that
  \begin{align*}
  \|x^{\overline{k}_2+1}-\overline{x}\|
  &\le\|x^{\overline{k}_2+1}-x^{\overline{k}_2}\|+\|x^{\overline{k}_2}-\overline{x}\|\stackrel{\eqref{xkp1-ineq41}}{\le} (1\!+\!\varsigma)\|y^{\overline{k}_2}\!-\!x^{\overline{k}_2}\|+\|x^{\overline{k}_2}\!-\!\overline{x}\|\\
  &\stackrel{\eqref{ykxk-bound}}{\le} (1\!+\!\varsigma)\eta{\rm dist}(x^{\overline{k}_2},\mathcal{X}^*)+\|x^{\overline{k}_2}-\overline{x}\|\le\big((1\!+\!\varsigma)\eta\!+\!1\big)\|x^{\overline{k}_2}\!-\!\overline{x}\|\le\varepsilon_2,  
 \end{align*}
 where the fourth inequality is due to ${\rm dist}(x^{\overline{k}_2},\mathcal{X}^*)\le\|x^{\overline{k}_2}-\overline{x}\|\le\widetilde{\varepsilon}\le1$. To proceed the induction, fix any $k>\overline{k}_2$ and assume that $x^{l+1}\in\mathbb{B}(\overline{x},\varepsilon_2)$ for all $\overline{k}_2\le l\le k\!-\!1$. By applying the given local error bound assumption, we can obtain that  
 \begin{align*}
  \|x^{k+1}\!-\!x^{\overline{k}_2}\|
  &\le{\textstyle\sum_{l=\overline{k}_2}^{k}}\|x^{l+1}-x^l\|
  \stackrel{\eqref{xkp1-ineq41}}{\le} (1\!+\!\varsigma){\textstyle\sum_{l=\overline{k}_2}^{k}}\|y^l\!-\!x^l\|\\
  &\stackrel{\eqref{ykxk-bound}}{\le} \eta(1\!+\!\varsigma){\textstyle\sum_{l=\overline{k}_2}^{k}}{\rm dist}(x^l,\mathcal{X}^*)
  \le \eta\kappa(1\!+\!\varsigma){\textstyle\sum_{l=\overline{k}_2}^{k}}[r(x^l)]^{\gamma}\\ 
  &\stackrel{\eqref{rk-recursion}}{\le} \eta\kappa(1\!+\!\varsigma)[r(x^{\overline{k}_2})]^{\gamma}{\textstyle\sum_{l=\overline{k}_2}^{k}}\widetilde{\sigma}^{l-\overline{k}_2}
  \le\frac{\eta\kappa(1\!+\!\varsigma)}{1-\widetilde{\sigma}}[r(x^{\overline{k}_2})]^{\gamma}\\
  &\stackrel{\eqref{rxk-ineq41}}{\le}\frac{\eta\kappa(1\!+\!\varsigma)}{1-\widetilde{\sigma}}(2\!+\!L_g)^{\gamma}\|x^{\overline{k}_2}-\overline{x}\|^{\gamma}\le\frac{\varepsilon_2}{2}.
  \end{align*}
  This means that $\|x^{k+1}-\overline{x}\|\le \varepsilon_2/2+\|x^{\overline{k}_2}-\overline{x}\|\le\varepsilon_2$. 
   
  Now we are in a position to prove the convergence result. As $\lim_{k\to\infty}r(x^k)=0$, for any $\varepsilon>0$, there exists $\widehat{k}\ge \overline{k}_2$ such that $r(x^k)\le\varepsilon$ for all $k\ge\widehat{k}$. Fix any $i_1>i_2\ge\widehat{k}$. Using the same arguments as those for the above inequality yields that
  \begin{align*}
   \|x^{i_1}-x^{i_2}\|\le\sum_{l=i_2}^{i_1-1}\|x^{l+1}-x^l\|
    \le\frac{(1+\varsigma)\eta\kappa}{1-\widetilde{\sigma}}[r(x^{i_2})]^{\gamma}
   \le\frac{(1+\varsigma)\eta\kappa}{1-\widetilde{\sigma}}\varepsilon^{\gamma},
  \end{align*}
  which implies that $\{x^k\}_{k\in\mathbb{N}}$ is a Cauchy sequence, so converges to $\overline{x}\in\mathcal{X}^*$. Next we deduce the convergence rate of $\{r(x^k)\}_{k\in\mathbb{N}}$ and $\{x^k\}_{k\in\mathbb{N}}$. Fix any $k>\widehat{k}$. Then $x^k\in\mathbb{B}(\overline{x},\varepsilon_2)$ follows \eqref{induction}, so \eqref{rk-rate} holds for any $k>\widehat{k}$. That is, the first inequality in \eqref{aim-rineq} holds. From the local error bound and the third inequality in \eqref{rk-rate},  
  \begin{equation}\label{dist-rate}
  {\rm dist}(x^{k+1},\mathcal{X}^*)\le\kappa[(3+\!L_g)\varrho \widehat{L}\!+\!{L_H}/{2}\!+\!\widehat{L}]^{\gamma}\eta^{\gamma(q-1)}{\rm dist}(x^k,\mathcal{X}^*)^{\gamma(q-1)}.
 \end{equation}
 Note that $\lim_{k\to\infty}{\rm dist}(x^k,\mathcal{X}^*)=0$ by the local error bound and $\lim_{k\to\infty}r(x^k)=0$. Inequality \eqref{dist-rate} implies that for any $k>\widehat{k}$, ${\rm dist}(x^{k+1},\mathcal{X}^*)\le\frac{1}{2}{\rm dist}(x^k,\mathcal{X}^*)$. Then, 
  \begin{align}\label{ineq-xk}
   \|x^{i_1}\!-\!x^{i_2}\|
   &\le{\textstyle\sum_{k=i_2}^{i_1-1}}\|x^{k+1}\!-\!x^k\|\stackrel{\eqref{xkp1-ineq41}}{\le}
    (1\!+\!\varsigma){\textstyle\sum_{k=i_2}^{i_1-1}}\|y^k\!-\!x^k\|\nonumber\\
   &\le(1\!+\!\varsigma)\eta{\textstyle\sum_{k=i_2}^{i_1-1}}{\rm dist}(x^k,\mathcal{X}^*)\nonumber\\
   &\stackrel{\eqref{ykxk-bound}}{\le}(1\!+\!\varsigma)\eta{\rm dist}(x^{i_2},\mathcal{X}^*){\textstyle\sum_{k=i_2}^{i_1-1}}(1/2)^{k-i_2}\nonumber\\
   &\le 2(1\!+\!\varsigma)\eta{\rm dist}(x^{i_2},\mathcal{X}^*)\quad\ {\rm for\ any}\ i_1>i_2>\widehat{k}.
  \end{align}
 Letting $i_2=k+1$ and passing the limit $i_1\to\infty$ to the both sides of \eqref{ineq-xk} results in
 \begin{align*}
 \|x^{k+1}\!-\!\overline{x}\|
 &\le 2(1\!+\!\varsigma)\eta{\rm dist}(x^{k+1},\mathcal{X}^*)\\
 &\stackrel{\eqref{dist-rate}}{\le}2(1\!+\!\varsigma)\eta\kappa[(3\!+\!L_g)\varrho \widehat{L}\!+\!L_H/2\!+\!\widehat{L}]^{\gamma}\eta^{\gamma(q-1)}{\rm dist}(x^k,\mathcal{X}^*)^{\gamma(q-1)}\\
 &\le2\eta\kappa(1\!+\!\varsigma)[(3\!+\!L_g)\varrho \widehat{L}+L_H/2+\widehat{L}]^{\gamma}\eta^{\gamma(q-1)} \|x^k\!-\!\overline{x}\|^{\gamma(q-1)}.
 \end{align*} 
 This shows that the second inequality in \eqref{aim-rineq} holds. The proof is completed.
 \end{proof}
\begin{remark}\label{remark-conv}
 Theorem \ref{Theorem-Lconverge} states that the sequence $\{x^k\}_{k\in\mathbb{N}}$ converges to a second-order stationary point with $Q$-superlinear rate of order $\gamma(q\!-\!1)$, which is specified as $2$ for $q=3$ and $\gamma\!=1$, under the strict continuity of $\nabla^2\!f$ on an open convex set $\mathcal{N}$ with $\mathcal{O}\supset\mathcal{N}\supset\omega(x^0)$ and the local H\"{o}lderian error bound of order $\gamma\in(\frac{1}{q-1},1]$ on $\mathcal{X}^*$. To the best of our knowledge, even for nonsmooth convex composite optimization, this is the first quadratic convergence rate of the iterate sequence for CR methods under a local Lipschitzian error bound, which is equivalent to the KL property of $F$ with exponent $1/2$ by \cite[Proposition 2 (i)]{LiuPanWY22} for convex composite optimization. 
\end{remark} 

\begin{theorem}\label{obj-superlinear}
 Fix any $q\in(2,3]$ and $\overline{x}\in\omega(x^0)$. Suppose that $F$ is a KL function, that the assumptions of Theorem \ref{Theorem-Lconverge} hold at $\overline{x}$ with $\gamma\in(\frac{1}{q-1},1]$, and that there exist $\widetilde{\varepsilon}\in(0,\widehat{\varepsilon})$ and $\widehat{\kappa}\in(0,\kappa^{-\frac{1}{\gamma}})$ such that $F(x)\ge F(\overline{x})-\frac{\gamma\widehat{\kappa}}{1+\gamma}[{\rm dist}(x,\mathcal{S}^*)]^{\frac{1+\gamma}{\gamma}}$ for all $x\in\mathbb{B}(\overline{x},\widetilde{\varepsilon})$.
 Then, for all sufficiently large $k$, the following inequality holds
\[
 F(x^{k+1})-\overline{F}\le\Big(\frac{2L_H}{3}\!+\!\frac{\beta}{\sqrt{2}}\Big)\eta^{q-1}\max\big\{(\varrho \widehat{L}\eta^{q-1}\!+\!\kappa\widehat{c}^{\gamma})^3,1\big\}[\beta_0(F(x^k)-\overline{F})]^{\gamma(q-1)},
\]
where $\beta_0>0$ is a constant and the constants $\eta$ and $\beta$ are the same as before. That is, the sequence $\{F(x^k)\}_{k\in\mathbb{N}}$ converges to $\overline{F}$ with $Q$-superlinear rate of order $\gamma(q\!-\!1)$.
\end{theorem}
 \begin{proof}
 From \cite[Remark 1 (b)]{LiuPanWY22}, $F(z^*)\le F(\overline{x})=\overline{F}$ for all $z^*\in\mathcal{S}\cap\mathbb{B}(\overline{x},\widetilde{\varepsilon})$ (if necessary by shrinking $\widetilde{\varepsilon}$).
 As $\mathcal{X}^*\subset\mathcal{S}^*$, the local error bound condition implies that ${\rm dist}(x,\mathcal{S}^*)\le\kappa[r(x)]^{\gamma}$ for all $x\in\mathbb{B}(\overline{x},\widetilde{\varepsilon})$, which by \cite[Lemma 1]{LiuPanWY22} is equivalent to the $\gamma$-subregularity of $\partial F$ at $(\overline{x},0)$ with modulus $\kappa$. Together with the given assumption on $F$ and \cite[Theorem 3.4]{Mordu15}, there exists $\beta_0>0$ such that for all $x\in\mathbb{B}(\overline{x},\widetilde{\varepsilon})$,
 \begin{equation}\label{growth}
  \beta_0(F(x)-\overline{F})\ge {\rm dist}(x,\mathcal{S}^*)^{\frac{1+\gamma}{\gamma}}.
 \end{equation}
 In addition, from the local error bound condition on $\mathcal{X}^*$, it is easy to argue by contradiction that $\mathcal{S}^*\cap\mathbb{B}(\overline{x},\widetilde{\varepsilon})=\mathcal{X}^*\cap\mathbb{B}(\overline{x},\widetilde{\varepsilon})$. This implies that 
 \begin{equation}\label{ineq-dist}
 {\rm dist}(x,\mathcal{X}^*)={\rm dist}(x,\mathcal{S}^*)\quad\ \forall x\in\mathbb{B}(\overline{x},\widehat{\varepsilon}/2).
 \end{equation}
 Indeed, fix any $x\in\mathbb{B}(\overline{x},\widetilde{\varepsilon}/2)$ and pick $x^*\in\Pi_{\mathcal{S}^*}(x)$, we have $\|x^*-\overline{x}\|\le 2\|x-\overline{x}\|\le\widetilde{\varepsilon}$. Then, $x^*\in\mathcal{S}^*\cap\mathbb{B}(\overline{x},\widetilde{\varepsilon})=\mathcal{X}^*\cap\mathbb{B}(\overline{x},\widetilde{\varepsilon})$, so ${\rm dist}(x,\mathcal{S}^*)\le{\rm dist}(x,\mathcal{X}^*)\le\|x-x^*\|={\rm dist}(x,\mathcal{S}^*)$. 
 From the proof of Theorem \ref{globalconv}, we have $\lim_{k\to\infty}w^k=\overline{w}$. 
 For each $k\in\mathbb{N}$, write $z^k:=y^k-v^k$ and $x^{k+1,*}:=\Pi_{\mathcal{X}^*}(x^{k+1})$. From the proof of Proposition \ref{prop-dkbound1}, $x^{k+1,*}\in\Gamma$ for all $k$ large enough. Then, for all sufficiently large $k$, it holds that 
 \begin{align}\label{z^k-bound}
 \|z^k-x^{k+1,*}\|&\le\|z^k-x^{k+1}\|+\|x^{k+1}-x^{k+1,*}\|
 \stackrel{\eqref{new-iter}}{\le}\|v^k\|+{\rm dist}(x^{k+1},\mathcal{X}^*)\nonumber\\
 &\stackrel{\eqref{inexact-cond},\eqref{dist-rate}}{\le}\varrho L_k\|y^k-x^k\|^{q-1}+\kappa\widehat{c}^{\gamma}{\rm dist}(x^{k},\mathcal{X}^*)^{\gamma(q-1)}\nonumber\\
 &\stackrel{\eqref{ykxk-bound}}{\le}(\varrho L_k\eta^{q-1}\!+\kappa\widehat{c}^{\gamma})\,{\rm dist}(x^{k},\mathcal{X}^*)^{\gamma(q-1)}\nonumber\\
 &\le(\varrho\widehat{L}\eta^{q-1}\!+\kappa\widehat{c}^{\gamma})\,{\rm dist}(x^{k},\mathcal{X}^*)^{\gamma(q-1)},
 \end{align} 
 where the last inequality is due to Proposition \ref{prop2-xk} (i).
 As $\lim_{k\to\infty}{\rm dist}(x^{k+1},\mathcal{X}^*)=0$, we have $x^{k+1,*}\in\mathbb{B}(\overline{x},\widetilde{\varepsilon})$, and then $F(x^{k+1,*})\le\overline{F}$ for all $k$ large enough.
 Together with \eqref{Phi-fun}, for all $k$ large enough,
 $\Phi(w^k)\!-\!\overline{F}\le F(y^k\!-\!v^k)-F(x^{k+1,*})$. Invoking the second inequality of Lemma \ref{lemma-relation} with $x^*=x^{k+1,*}$ and \eqref{z^k-bound}, for all $k$ large enough, 
 \begin{align*}
  F(x^{k+1})-\overline{F}&\stackrel{\eqref{new-iter}}{\le}\Phi(w^k)-\overline{F}\le F(y^k\!-\!v^k)-F(x^{k+1,*})\\
  &\le (2L_H/3)\|z^k\!-\!x^{k+1,*}\|^3+({\beta}/{\sqrt{2}})\|y^k\!-\!x^k\|^{q-1}\|z^k\!-\!x^{k+1,*}\|\\
  &\stackrel{\eqref{ykxk-bound}}{\le}(2L_H/3)\|z^k\!-\!x^{k+1,*}\|^3+({\beta}/{\sqrt{2}})\eta^{q-1}{\rm dist}(x^{k},\mathcal{X}^*)^{q-1}\|z^k\!-\!x^{k+1,*}\|\\
  &\stackrel{\eqref{z^k-bound}}{\le}(2L_H/3)(\varrho \widehat{L}\eta^{q-1}\!+\kappa\widehat{c}^{\gamma})^3{\rm dist}(x^{k},\mathcal{X}^*)^{3\gamma(q-1)}\nonumber\\
  &\qquad+({\beta}/{\sqrt{2}})\eta^{q-1}(\varrho \widehat{L}\eta^{q-1}\!+\kappa\widehat{c}^{\gamma}){\rm dist}(x^{k},\mathcal{X}^*)^{(1+\gamma)(q-1)}\\
  &\le\Big(\frac{2L_H}{3}\!+\!\frac{\beta}{\sqrt{2}}\Big)\eta^{q-1}\max\big\{(\varrho \widehat{L}\eta^{q-1}\!+\!\kappa\widehat{c}^{\gamma})^3,1\big\}{\rm dist}(x^{k},\mathcal{X}^*)^{(1+\gamma)(q-1)}\\
  &\stackrel{\eqref{ineq-dist}}{=}\Big(\frac{2L_H}{3}\!+\!\frac{\beta}{\sqrt{2}}\Big)\eta^{q-1}\max\big\{(\varrho \widehat{L}\eta^{q-1}\!+\!\kappa\widehat{c}^{\gamma})^3,1\big\}{\rm dist}(x^k,\mathcal{S}^*)^{(1+\gamma)(q-1)}\\
  &\stackrel{\eqref{growth}}{\le}\Big(\frac{2L_H}{3}\!+\!\frac{\beta}{\sqrt{2}}\Big)\eta^{q-1}\max\big\{(\varrho \widehat{L}\eta^{q-1}\!+\!\kappa\widehat{c}^{\gamma})^3,1\big\}[\beta_0(F(x^k)-\overline{F})]^{\gamma(q-1)}.
 \end{align*} 
 This shows that the desired inequality holds. The proof is completed.
 \end{proof}
 \begin{remark}\label{remark-objrate}
 {\bf(a)} Theorem \ref{obj-superlinear} shows that the objective value sequence has the same local convergence rate as the iterate sequence, but it requires the local H\"{o}lderian error bound to hold at $\overline{x}\in\omega(x^0)$ for which $F$ satisfies the lower second-order growth. This condition automatically holds when $\overline{x}$ is locally optimal.

 \noindent
 {\bf(b)} Theorem \ref{obj-superlinear} with $q\!=3$ and $\gamma\!=1$ implies that for convex composite optimization problem \eqref{prob}, the objective value sequence $\{F(x^k)\}_{k\in\mathbb{N}}$ of our inexact CR method converges to $\overline{F}$ with $Q$-quadratic rate under a local Lipschitzian error bound. This recovers the conclusion of \cite[Theorem 1]{Doikov22} for $p=2$ by replacing the strong convexity assumption there on $F$ with a much weaker local error bound.
 \end{remark}

 \subsection{Complexity bounds}
 We derive a global complexity bound for Algorithm \ref{qregPNT} with $q\in[2,3]$ in terms of the first-order optimality conditions by replacing Assumption \ref{ass2} (ii) with the strict continuity of $\nabla^2\!f$ on a compact convex set $\Omega$ containing $\{(x^k,y^k)\}_{k\in\mathbb{N}}$. Note that $\Omega$ does not necessarily contain the compact convex set $\Gamma$.
 \begin{theorem}\label{complexity}
 Suppose that Assumptions \ref{ass1}-\ref{ass2} (i) holds, and that $\nabla^2\!f$ is strictly continuous on a compact convex set $\Omega$ containing $\{x^k\}_{k\in\mathbb{N}}\cup\{y^k\}_{k\in\mathbb{N}}$. 
 Then, for any given $\epsilon\in(0,1)$, $r(x^{K+1})\le\epsilon$ with $K\!:=\!\lfloor q(\sigma L_{m})^{-1}\min\{\widehat{\eta}^{\frac{q}{2}},\widehat{\eta}^{\frac{q}{q-1}}\}(F(x^0)\!-\!\overline{F})\epsilon^{-\frac{q}{q-1}}\rfloor$ for $\widehat{\eta}\!=\!(\varrho(3\!+\!L_g)\!+\!1)\widehat{L}\!+\!{L_H}/{2}$, where $L_H$ is the Lipschitz constant of $\nabla^2\!f$ on $\Omega$.
 \end{theorem}
 \begin{proof}
 Since $\nabla^2\!f(\cdot)$ is assumed to be strictly continuous on the compact convex set $\Omega$, it is Lipschitz continuous on $\Omega$ by \cite[Theorem 9.2]{RW98}. We first derive an upper bound of $r(x^{k+1})$ in terms of $\|y^k\!-\!x^k\|$. It is easy to check \eqref{ineq0-rk} holds for all $k\in\mathbb{N}$. From the Lipschitz continuity of  $\nabla^2\!f(\cdot)$ on $\Omega$ and $\{(x^{k},y^k)\}_{k\in\mathbb{N}}\subset\Omega$, inequality \eqref{ineq-quad} holds for all $k\in\mathbb{N}$. Then, for all $k\in\mathbb{N}$, it holds that
 \[
   \|\nabla\!f(x^{k+1})\!-\!\nabla\vartheta_k(y^k)\|\le L_g\|v^k\|+(L_H/2)\|y^k\!-\!x^k\|^2.
 \]
 By combining this inequality with \eqref{ineq0-rk} and \eqref{inexact-cond}, for each $k\in\mathbb{N}$, it holds that 
 \begin{align}\label{ineq-complexity}
  r(x^{k+1})&\le 3\|v^k\|+ L_g\|v^k\|+(L_H/2)\|y^k\!-\!x^k\|^2+L_k\|y^k\!-\!x^k\|^{q-1}\nonumber\\
 &\le(\varrho(3\!+\!L_g)\!+\!1)\widehat{L}\|y^k\!-\!x^k\|^{q-1}+(L_H/2)\|y^k\!-\!x^k\|^2,
 \end{align}
 where the second inequality is due to  $L_k\le\widehat{L}$ by Proposition \ref{prop2-xk} (i). 
 As $r(x^{K+1})\le\epsilon$ and Algorithm \ref{qregPNT} stops at the $(K\!+\!1)$th iteration, we have $r(x^i)>\epsilon$ for $i=1,\ldots,K$. Fix any $i\in\{1,\ldots,K\}$. When $\|y^{i-1}\!-\!x^{i-1}\|>1$, combining the above \eqref{ineq-complexity} with $r(x^i)>\epsilon$ and the definition of $\widehat{\eta}$ leads to $\|y^{i-1}\!-\!x^{i-1}\|\ge \widehat{\eta}^{-1/2}\epsilon^{1/2}\ge \widehat{\eta}^{-1/2}\epsilon^{1/(q-1)}$. When $\|y^{i-1}\!-\!x^{i-1}\|\le1$, the above \eqref{ineq-complexity} along with  $r(x^i)>\epsilon$ and the definition of $\widehat{\eta}$ implies that $\|y^{i-1}\!-\!x^{i-1}\|\ge \widehat{\eta}^{-1/(q-1)}\epsilon^{1/(q-1)}$. Thus, $\|y^{i-1}\!-\!x^{i-1}\|\ge \min\{\widehat{\eta}^{-1/2},\widehat{\eta}^{-1/(q-1)}\}\epsilon^{1/(q-1)}$. Together with Proposition \ref{prop1-xk} (i), it follows that 
 \begin{equation*}
  F(x^{i-1})-F(x^{i})\ge\sigma L_{m}q^{-1}\|y^{i-1}\!-\!x^{i-1}\|^q\ge \sigma L_{m}q^{-1}\min\{\widehat{\eta}^{-q/2},\widehat{\eta}^{-q/(q-1)}\}\epsilon^{\frac{q}{q-1}}.
 \end{equation*}
  Recall that $F(x^K)$ is bounded below by $\overline{F}$ by Proposition \ref{prop1-xk} (i). Then, it holds that
  \[
    F(x^0)-\overline{F}\ge{\textstyle\sum_{i=1}^{K}}(F(x^{i-1})\!-\!F(x^{i}))\ge K\sigma L_{m}q^{-1}\min\{\widehat{\eta}^{-q/2},\widehat{\eta}^{-q/(q-1)}\}\epsilon^{\frac{q}{q-1}},
 \]
 which, after a suitable rearrangement, yields the desired result.
 \end{proof}

 When $g\equiv0$, the complexity bound of Theorem \ref{complexity} with $q=3$ matches the one obtained in \cite{Nesterov06} for the CR method and those obtained in \cite{Cartis11b} for its variants, namely, $O(\epsilon^{-\frac{3}{2}})$. When $g=\chi_{C}$ for a closed convex set $C\subset\mathbb{X}$, the complexity bound of Theorem \ref{complexity} also fits with the one established in \cite[Theorem 3.12]{Cartis18}.

 \section{Numerical experiments}\label{sec5}

 We test the performance of Algorithm \ref{qregPNT} by applying it to problem \eqref{prob} with highly nonlinear $f$ and common nonsmooth convex $g$. All tests are performed on a desktop running on Matlab R2020b and 64-bit Windows System with an Intel(R) Core(TM) i9-10850K CPU 3.60GHz and 32.0 GB RAM.    
\subsection{Implementation of Algorithm \ref{qregPNT}}\label{sec5.1}

 The core of Algorithm \ref{qregPNT} is the inexact computation of subproblem \eqref{subprobj}, the minimization of the sum of the $q$-order regularized quadratic approximation of $f$ at $x^k$ and the nonsmooth convex $g$. Based on Remark 3.3 in the archive version of this work, we choose ZeroFPR as a solver to subproblems. ZeroFPR is a line-search algorithm proposed in \cite{Themelis18} for the composite problem of form \eqref{prob} with a general nonsmooth $g$ by seeking a critical point of the forward-backward envelope of $F$ with the limited-memory BFGS (lbfgs). When calling ZeroFPR to solve \eqref{subprobj}, we adopt its default setting except that the inexactness criterion \eqref{inexact-cond} is used as the stopping condition and the maximum number of iterations is set as $10^3$. For the subsequent tests, the parameters of Algorithm \ref{qregPNT} are chosen as 
 \[
  \varrho={0.9}/{(1.1\!+\!L_{0})},\ \tau=10,\ \delta=10^{-5},\ L_{m}=10^{-12}\ \ {\rm and}\  \ L_{\!M}=10^{8},
\]
 where $L_{0}$ is an estimate for the constant $L_g$ appearing in \eqref{grad-Lip}, i.e., a positive constant such that $f(x^{0,*})\le f(x^{0})+\langle\nabla\!f(x^0),x^{0,*}\!-\!x^0\rangle\!+\!\frac{L_{0}}{2}\|x^{0,*}\!-\!x^0\|^2$ holds with $x^{0,*}\!=\!{\rm prox}_{L_{0}^{-1}g}(x^0\!-\!L_0^{-1}\nabla\!f(x^0))$. Notice that the constant $L_{k,0}$ in step 1 of Algorithm \ref{qregPNT} is an initial upper estimation of the H\"older modulus of exponent $q-2$ for $\nabla^2\!f$ at $x^k$. Motivated by the Barzilai-Borwein rule \cite{Barzilai88} for estimating the Lipschitz constant of $\nabla\!f$ (see \cite{Wright09}), we always choose $L_{0,0}=L_0$, and $L_{k,0}$ for $k\ge 1$ by the following formula
 \begin{equation*}
 \max\Big\{\min\Big\{\frac{\|[\nabla^2\!f(x^{k})-\nabla^2\!f(x^{k-1})](x^k\!-\!x^{k-1})\|}{\|x^{k}-x^{k-1}\|^{q-1}},L_{\!M}\Big\},L_{m}\Big\}.
 \end{equation*}
\subsection{Validation of local convergence rate}\label{sec5.2}

 We validate the local superlinear convergence rate and global complexity bound results of Algorithm \ref{qregPNT}. Figure \ref{fig_rate} (a) plots the KKT residual curves yielded by Algorithm \ref{qregPNT} with different $q\in(2,3]$ for solving the $\ell_1$-norm regularized logistic regression with dataset ``leu''. For this class of convex composite problems, Section \ref{sec5.3.1} contains its detailed description, and the local Lipschitzian error bound was shown to hold at any $\overline{x}\in\mathcal{S}^*=\mathcal{X}^*$ in \cite{ZhouSo17}. We see that the KKT residual sequence returned by Algorithm \ref{qregPNT} with a larger $q\in(2,3]$ indeed has a better $Q$-superlinear rate and global complexity bound, which coincides with the theoretical results of Theorems \ref{Theorem-Lconverge} and \ref{complexity}, but the times listed in the legend indicates that Algorithm \ref{qregPNT} with a larger $q\in(2,3]$ may not require less one. Figure \ref{fig_rate} (b) plots the residual curves returned by Algorithm \ref{qregPNT} with different $q\in(2,3]$ for solving a test example of portfolio decision with higher moments from Section \ref{sec5.3.3}. For this class of nonconvex composite problems, now it is unclear whether the local H\"olderian error bound in Theorems \ref{Theorem-Lconverge} and \ref{complexity} or even the one in \eqref{ebound-MS} hold. 
 The curves of Figure \ref{fig_rate} (b) show that the sequence produced by Algorithm \ref{qregPNT} with $q\in(2,3]$ has no superlinear rate, but Algorithm \ref{qregPNT} with a larger $q$ still has a better global complexity bound and requires less running time by those in the legend. 
\begin{figure}[h]
 \centering
 \subfigure[\label{figPR1a}]{\includegraphics[scale=0.4]{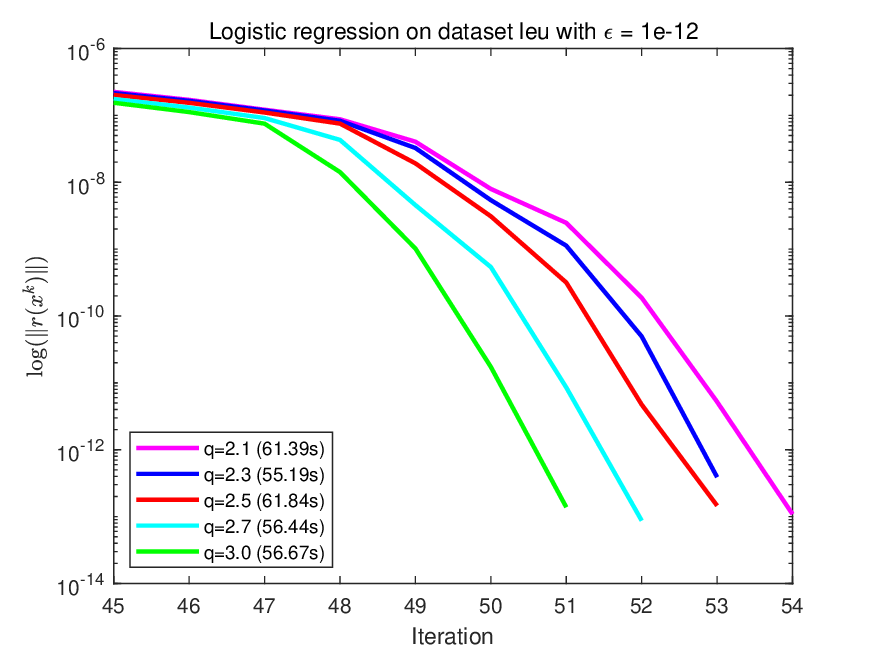}}
 \subfigure[\label{figPR1b}]{\includegraphics[scale=0.4]{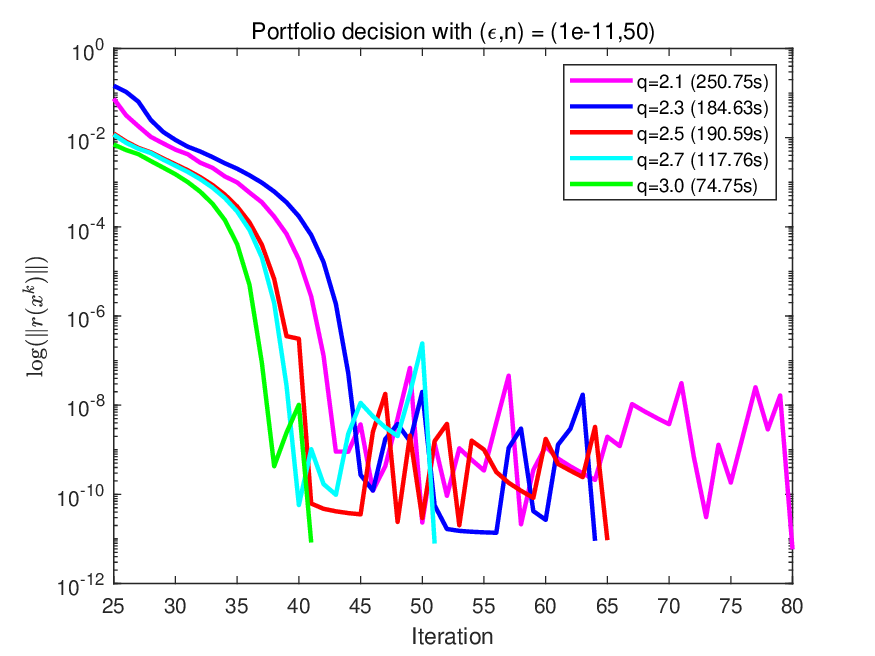}}
 \setlength{\abovecaptionskip}{2pt}
 \setlength{\belowcaptionskip}{2pt}
 \caption{Convergence behavior of the sequences produced by Algorithm \ref{qregPNT} with different $q$}
 \label{fig_rate}
\end{figure}
\subsection{Numerical comparisons}\label{sec5.3}

Based on the convergence curves in Figure \ref{fig_rate}, we test the performance of Algorithm \ref{qregPNT} with $q=3$ (i.e. the CR method), and compare its performance with that of ZeroFPR and IRPNM on three classes of problems of form \eqref{prob} to validate its efficiency. As mentioned above, ZeroFPR is a first-order line-search algorithm that is applicable to \eqref{prob}, and its code can be downloaded from \url{http://github.com/kul-forbes/ForBES}. IRPNM is an inexact proximal Newton method proposed in \cite{LiuPanWY22}, whose each iteration involves minimizing the sum of a strongly convex quadratic approximation of $f$ and the nonsmooth convex function $g$ with the solver dALMSN, a dual augmented Lagrangian method armed with the semismooth Newton. The numerical results reported in \cite{LiuPanWY22} and \cite{Dahl24} indicate that IRPNM is superior to GIPN, an inexact proximal Newton-type method proposed in \cite{Kanzow21} by solving subproblem with FISTA (named GIPN-F) or the KKT system with the semismooth Newton method (named GIPN-S), and AC-FISTA proposed in \cite{Liang23}. Here, we do not compare the performance of our CR method with that of GIPN and AC-FISTA separately. For ZeroFPR and IRPNM to solve problem \eqref{prob}, we adopt their default setting. For fairness of numerical comparisons, the three solvers start from the same initial point $x^0\in{\rm dom}\,g$ and stop under the same condition $r(x^k)\le\epsilon$.
\subsubsection{$\ell_1$-norm regularized logistic regression}\label{sec5.3.1}

The $\ell_1$-norm regularized logistic regression takes the form of \eqref{prob} with $\mathbb{X}=\mathbb{R}^n,\,f(\cdot)=\psi(A\,\cdot)$ and $g(\cdot)=\lambda\|\cdot\|_1$, where $\lambda>0$ is a regularization parameter, $\psi(u)=\frac{1}{m}\sum_{i=1}^m\log(1+\exp(-b_iu_i))$ with $b_1,\ldots,b_m\in\{-1,1\}$ being the labels, and $A\in\mathbb{R}^{m\times n}$ is a data matrix. We test this class of convex composite problems with the data $(A,b)$ from LIBSVM; see \url{https://www.csie.ntu.edu.tw/~cjlin/libsvmtools/datasets/}.  The dimension of test instances and the largest eigenvalue of $AA^{\top}$ are listed in Table \ref{table_info}.
\begin{table}[h]
 \centering
 \setlength{\tabcolsep}{2pt}
 \setlength{\belowcaptionskip}{2pt}
 \caption{Test data for $\ell_1$-norm regularized logistic regression}
 \label{table_info}
 \begin{tabular}{cc|cc||cc|cc}
  \hline
  & data & $\lambda_{\max}(AA^{\top})$ & $(m,n)$ & & data & $\lambda_{\max}(AA^{\top})$ & $(m,n)$\\\hline
  I & colon-cancer & 1.93e+4 &	 (62,2000) & IV &  leu & 6.68e+4 &	 (38,7129)  \\ \hline
  II & news20 	& 1.17e+3 &	 (19996,1355191) & V & gisette  &  3.37e+6 & (1000,5000)\\ \hline
  III & duke 	& 1.11e+5 &	 (44,7129) & VI & arcene 	 & 1.45e+10 & (100,10000)\\ \hline
 \end{tabular}
\end{table}

 We report the numerical results of three solvers with the stopping accuracy $\epsilon=10^{-6}$ and the starting point $x^{0}=0$ in Table \ref{table_logisitic}, where the ``resi'' column lists the final KKT residuals, and the ``${\rm xnz}$'' column reports the number of nonzeros of the final iterate, defined by $\min\big\{k\ |\ \sum_{i=1}^k|x^f|_i^{\downarrow}\ge 0.999\|x^f\|_1\big\}$, where $x^f$ is the final iterate, and $|x^f|^{\downarrow}$ is obtained by sorting $|x^f|$ in a descending order. We see that the CR method returns optimal solutions of all test problems within $54$ steps except ``arcene'' with $\lambda=10^{-4}$ for the maximum number of iterations $10^2$, and requires less iterations than IRPNM for the test problems with $\lambda=10^{-6}$ that are more difficult than their counterparts with $\lambda=10^{-4}$. ZeroFPR requires the maximum number of iterations $10^4$ for all test problems except ``colon-cancer'' and ``news20'' with $\lambda=10^{-4}$, and IRPNM also requires the maximum number of iterations $10^3$ for ``arcene''.
 We need to point out that the results of IRPNM for ``leu'' with $\lambda=10^{-4}$ and ``arcene'' are obtained by choosing again the initial penalty parameter of dALMSN. This precisely reflects the attractive convergence rate and more robustness of the CR method. In addition, Table \ref{table_logisitic} also indicates that the solutions returned by the CR method have better objective values and KKT residuals, though it needs more running time.

\begin{table}[h]
 \centering
 \setlength{\tabcolsep}{1.5pt}
 \setlength{\belowcaptionskip}{2pt}
 \caption{Numerical comparisons on $\ell_1$-norm regularized logistic regressions with real data}
	\label{table_logisitic}
 \scalebox{0.8}{
	\begin{tabular}{@{\extracolsep{\fill}}c|c|ccccc|ccccc|ccccc@{\extracolsep{\fill}}}
		\hline
		\multicolumn{1}{c|}{} & \multicolumn{1}{c|}{} & \multicolumn{5}{c|}{CR method} &\multicolumn{5}{c|}{ZeroFPR} & \multicolumn{5}{c}{IRPNM} \\
		\hline
		data & $\lambda$ &  iter & Fval & resi & time & xnz &  iter & Fval & resi & time & xnz&  iter & Fval & resi & time & xnz\\
		\hline
		\multirow{2}{*}{I} & 1e-4
		& 31 & 3.23e-3 & 8.78e-7 & 4.2 	& 30
		& 8984 & 3.23e-3 & 9.98e-7 & 3.8  & 30 
		& 24 & 3.23e-3 & 9.98e-7 & 0.5 & 30\\ 
		\multirow{2}{*}{} & 1e-6
            & 30 & 5.14e-5 & 8.09e-7 & 4.1 & 57
		& -- & 5.31e-5 & 1.06e-5 & 3.9 & 103
		& 36 & 5.32e-5 & 9.95e-7 & 0.7 &   88\\
		\hline
		\multirow{2}{*}{II} & 1e-4 
		& 10 & 4.04e-1 & 3.54e-7 & 118.6 &  277
		& 418 & 4.04e-1 & 3.90e-7 & 41.0 & 277
		& 9 & 4.04e-1 & 9.93e-7 & 95.6 &  277\\
		\multirow{2}{*}{} & 1e-6
		& 20 & 4.00e-2 & 7.12e-7 & 1009.0 & 3168
		& -- & 4.00e-2 & 5.86e-6 & 997.8 & 2898
		& 342 & 4.01e-2 & 1.00e-6 & 2617.7 & 3276\\ 
		\hline
        \multirow{2}{*}{III} & 1e-4
            & 33 & 1.98e-3 & 8.08e-7 & 18.0 &  29
		& -- & 2.01e-3 & 3.57e-4 & 8.3 & 78
		& 16 & 1.98e-3 & 9.97e-7 & 0.7 &  29\\ 
		\multirow{2}{*}{} & 1e-6
            & 30 & 3.08e-5 & 9.40e-7 & 14.4 & 81
		& -- & 3.42e-5 & 1.64e-5 & 8.0 & 346
		& 47 & 3.19e-5 & 1.00e-6 & 1.5 & 139\\		
		\hline
		\multirow{2}{*}{IV} & 1e-4
            & 44 & 1.71e-3 & 9.57e-7 & 15.1 & 26
		& -- & 1.73e-3 & 2.14e-5 & 8.0 & 32
            & 30 & 1.71e-3 & 9.79e-7 & 1.8 & 26
		\\ 
		\multirow{2}{*}{} & 1e-6
		& 31  & 2.69e-5 & 8.42e-7 & 12.2 & 74
		& -- & 3.01e-5 & 1.29e-5 & 7.6 & 334
		& 40 & 2.83e-5 & 1.00e-6 & 1.5 &  150\\		
		\hline
		\multirow{2}{*}{V} & 1e-4 
            & 34 & 8.71e-3 & 9.83e-7 & 525.8 & 246
		& -- & 8.73e-3 & 1.85e-4 & 113.6 & 298
		& 16  & 8.71e-3 & 9.97e-7 & 21.7 & 246\\
		\multirow{2}{*}{} & 1e-6
		& 34 & 1.51e-4 & 9.00e-7 & 363.9 &  325
		& -- & 1.56e-4 & 1.01e-5 & 112.8 & 499
		& 34 & 1.53e-4 & 1.00e-6 & 36.7 & 399\\		
		\hline
		\multirow{2}{*}{VI} & 1e-4
		& -- & 8.70e-5 & 5.40e-5 & 95.3 & 140
		& -- & 9.66e-5 & 3.15e-4 & 7.8 & 577
		& -- & 9.62e-5 & 2.62e-4 & 36.7 	& 527\\
		\multirow{2}{*}{} & 1e-6
            & 54  & 1.29e-6 & 9.28e-7 & 26.5 & 229
		& -- & 1.34e-6 & 4.23e-6 & 7.8 & 620
		& -- & 2.62e-6 & 5.59e-5 & 37.3 & 7225\\
		\hline
\end{tabular}}
{\it\small where ``$-$'' means that the maximum number of iterations is attained}
\end{table}

\subsubsection{$\ell_1$-norm regularized Student's $t$-regressions}\label{sec5.3.2}
This class of problem takes the form of \eqref{prob} with $\mathbb{X}=\mathbb{R}^n,\, f(\cdot)\!=\psi(A\,\cdot\!-\!b)$ and $g(\cdot)\!=\lambda\|\cdot\|_1$, where $\psi(u)\!:=\!\sum_{i=1}^m\log(1+{u_i^2}/{\nu})\ (\nu>0)$ for $u\in\mathbb{R}^m$.
Such a nonconvex loss function $f$ was introduced in \cite{Aravkin12} to deal with the data contaminated by heavy-tailed Student-$t$ noise. The test examples are randomly generated in the same way as in \cite{Milzarek14}. Specifically, we generate a true sparse signal $x^{\rm true}$ of length $n=512^2$ with $s=\lfloor\frac{n}{40}\rfloor$ nonzero entries whose indices are chosen randomly, and then compute each nonzero component by the formula $x^{\rm true}_i=\eta_1(i)10^{d\eta_2(i)/20}$, where $\eta_1(i)\in\{1,-1\}$ is a random sign and $\eta_2(i)$ is uniformly distributed in $[0,1]$. The signal has a dynamic range of $d$ dB with $d\in\{20,40,60,80\}$. The matrix $A\in\mathbb{R}^{m\times n}$ takes $m={n}/{8}$ random cosine measurements, i.e., $Ax = (\verb"dct"(x))_J$, where $\verb"dct"$ is the discrete cosine transform and $J\subseteq\{1,2,\ldots,n\}$ with $|J|=m$ is an index set chosen randomly. The vector $b$ is obtained by adding Student's $t$-noise with degree of freedom $4$ and rescaled by $0.1$ to $Ax^{\rm true}$. We choose $\nu=0.25$ and $\lambda=c_{\lambda}\|\nabla\!f(0)\|_{\infty}$ with $c_{\lambda}>0$ for numerical tests.

 We run three solvers from the starting point $x^{0}\!=A^{\top}b$ with the stopping accuracy $\epsilon=10^{-5}$ for $10$ independent trials, i.e., ten groups of data $(A,b)$ generated randomly. Table \ref{table_dct} reports the average objective values, KKT residuals, and running time over $10$ independent trials. We see that the three solvers yield the same objective values for all test examples, and our CR method and IRPNM return comparable KKT residuals for all test examples, which are superior to those returned by ZeroFPR. Compared with IRPNM, our CR method requires fewer iterations and needs less running time for $d=80$ but a little more time for other $d$. Notice that the test examples with $d=80$ are more difficult than those with $d=20,40$ and $60$. Moreover, after checking, we find that our CR method returns second-order stationary points for all test examples except for the case where $d=20, c_{\lambda}=0.1$. Among others, the $\lambda_{\rm min}$ column in Table \ref{table_dct} reports the smallest one among the minimum eigenvalues of $[\nabla^2\!f(x^{*})]$ for $10$ trials. 
 This agrees with the fact that the CR method has a better local superlinear convergence rate and global complexity bound in theory.
\begin{table}[h]
 \centering
 \setlength{\tabcolsep}{2pt}
 \setlength{\belowcaptionskip}{2pt}
 \caption{Numerical comparisons on regularized Student's $t$-regression with $n=512^2$ and $m=n/8$}
 \label{table_dct}
 \scalebox{0.8}{
 \begin{tabular}{c|c|ccccc|cccc|cccc}
 \hline
 \multicolumn{2}{c|}{} &\multicolumn{5}{c|}{CR method} & \multicolumn{4}{c|}{ZeroFPR} & \multicolumn{4}{c}{IRPNM}  \\
 \hline
 $d$ & $c_{\lambda}$ & Iter & Fval & resi & time & $\lambda_{\rm min}$ & Iter & Fval & resi & time & Iter & Fval & resi & time\\
 \hline
 \multirow{2}{*}{20} & {0.1}  
  & 10.4 & 9.53e+3 & 8.24e-6 & 10.2 & -1
  & 147.8 & 9.53e+3 & 9.15e-6 & 9.3 
  & 24.6  & 9.53e+3 & 8.80e-6 & 9.3
  \\ 
  \multirow{2}{*}{} & {0.01}
  & 5.7 & 1.02e+3 & 8.71e-6 & 55.2 & 0
  & 1014.0 & 1.02e+3 & 9.56e-6 & 62.8
  & 10.8 & 1.02e+3 & 8.30e-6 & 25.0  
  \\ \hline 		
  \multirow{2}{*}{40} & {0.1}
  & 10.6  & 2.38e+4 & 8.82e-6 & 37.0 &  0 
  & 440.8 & 2.38e+4 & 9.26e-6 & 27.6 
  & 17.5 & 2.38e+4 & 8.40e-6 & 24.6
  \\ 
	\multirow{2}{*}{} & {0.01}
        & 6.3 & 2.40e+3 & 8.07e-6 & 133.8 & 0
        & 3255.9 & 2.40e+3 & 9.63e-6 & 202.6 
        & 14.2 & 2.40e+3 & 7.24e-6 & 81.6
	\\\hline
	\multirow{2}{*}{60} & {0.1}
        & 7.0 & 5.42e+4 & 8.89e-6 & 87.5  & 0
        & 1544.5 & 5.42e+4 & 9.48e-6 & 95.9 
        & 23.2 & 5.42e+4 & 8.09e-6 & 55.1  
	\\
	\multirow{2}{*}{} & {0.01}
        & 9.2 & 5.42e+3 & 8.77e-6 & 255.1 & 0
        & 6272.6 & 5.42e+3 & 9.62e-6 & 390.0 
        & 17.6 & 5.42e+3 & 8.21e-6 & 178.2
        \\\hline
	\multirow{2}{*}{80} & {0.1}
	& 8.0  & 1.35e+5 & 8.51e-6 & 196.6 & 0
        & 5277.2 & 1.35e+5 & 9.51e-6 & 329.9
        & 109.6 & 1.35e+5 & 8.35e-6 & 233.6
        \\
        \multirow{2}{*}{} & {0.01}
        & 18.2 & 1.35e+4 & 8.77e-6 & 635.2 & 0
        & 10000 & 1.35e+4 & 2.99e-4 & 624.2
        & 116.4 & 1.35e+4 &  6.92e-6 & 794.5
        \\ \hline
\end{tabular}}
\end{table}
\vspace{-10pt}

\subsubsection{Portfolio decision with higher moments}\label{sec5.3.3}
In this example, we focus on the higher moment portfolio selection problem described in \cite{Dihn11,Zhou21}. Let $\xi$ be a random vector to denote the return rate of $n$ assets. It is known that the mean and variance of $\xi$ have been widely used to assess the  expected profit and risk of the portfolio, and the associated standard mean-variance portfolio model is tailored to Gaussian distributed return. When the return distribution exhibits heavy tail and asymmetry, it is reasonable to extend the classical mean-variance optimization to higher moments \cite{Jean71} and consider the MVSK model of form \eqref{prob} with $\mathbb{X}=\mathbb{R}^n$ and
\begin{equation*}
 f(x)=-\omega_1E(x)+\omega_2V(x)-\omega_3S(x)+\omega_4K(x)\ \ {\rm and}\ \ g(x)=\chi_{\Delta}(x),
\end{equation*}
where $\omega_i\in[0,1]$ with $\sum_{i=1}^4\omega_i=1$ are the investor's preference levels corresponding to the four moments, and $\Delta=\{x\in\mathbb{R}_{+}^n\ |\ \sum_{i=1}^nx_i=1\}$ denotes the feasible set of portfolio weights. Here, $E,V,S,K$ represent four moments of the portfolio, and for their definitions and gradient and Hessian computation formulas, please refer to \cite{Zhou21}. 

We choose the weekly returns $R\in\mathbb{R}^{N\times T}$ of NASDAQ100 index during the period 26/08/2022-25/08/2023 as the test dataset, where $N$ is the total number of stocks and $T$ is the number of weeks during the selected period. The test data is obtained by the following two steps: {\bf(1)} randomly select $n\le N$ stocks from the dataset and compute the return rate of each selected stock via $\xi_{it}=100[\log(R_{i,t+1})-\log(R_{i,t})]$ for $t=1,2,\ldots,T\!-\!1$; {\bf(2)} compute four sample moments of the selected $n$ stocks as in \cite{Dihn11} by using the return rate and the weight vector $(\omega_1,\omega_2,\omega_3,\omega_4)=(0.29,0.21,0.4,0.1)$. Table \ref{table_portfolio} reports the numerical results of three solvers with the stopping accuracy $\epsilon=10^{-5}$ and the starting point $x^0=\frac{1}{n}(1,1,\ldots,1)^{\top}$. We see that the CR method is significantly superior to ZeroFPR in terms of running time, and it returns the comparable objective values and a little better KKT residuals for all test examples. For this class of problems, the computation of function and gradient values of $f$ is expensive. The time difference between CR method and ZeroFPR is attributed to the fact that Algorithm \ref{qregPNT} needs fewer calls of function and gradient values of $f$. Our CR method has the comparable performance with IRPNM except that the former needs a little fewer iterations and returns a little better objective values. For this class of test problems, IRPNM with the default setting exhibits poor numerical performance, failing to return a solution with the desired accuracy within the maximum number of iterations. The results of IRPNM in Table \ref{table_portfolio} are obtained by modifying IRPNM with a tighter approximation as suggested in \cite[Section 4.2]{LiuPanWY22} to construct the subproblem as well as choosing the initial penalty parameter of dALMSN again.
\begin{table}[H]
 \centering
 \setlength{\tabcolsep}{2pt}
 \setlength{\belowcaptionskip}{2pt}
 \caption{Numerical comparisons on portfolio decision with higher moments}
 \label{table_portfolio}
 \scalebox{0.8}{
 \begin{tabular}{c|ccccc|ccccc|ccccc}
 \hline
 \multicolumn{1}{c|}{}  &\multicolumn{5}{c|}{CR method} & \multicolumn{5}{c|}{ZeroFPR} & \multicolumn{5}{c}{IRPNM}\\
 \hline
  $n$  & iter & Fval &  resi & time & xnz & iter & Fval &  resi & time & xnz & iter & Fval &  resi & time & xnz\\
  \hline
  50
  & 33.6 & -4.49e-2 & 4.69e-6 & 5.4 & 18.1
  & 3396.8 & -4.41e-2 & 9.11e-6 & 12.5 & 17.8
  & 33.4 & -4.36e-2 & 9.11e-6 & 0.8 &	 18.3\\ 
  \hline
  80
  & 37.7 & -9.13e-2 & 2.78e-6 & 9.6 & 22.3
  & 4748.6 & -9.10e-2 & 1.59e-3 & 128.9 & 22.1
  & 43.5 & -8.91e-2 & 9.56e-6 & 4.4 & 22.8\\
  \hline
 100
 & 35.8 & -1.12e-1 & 3.47e-6 & 15.3 & 26.0
 & 5450.3 & -1.14e-1 & 1.29e-5 & 348.9 & 25.3
 & 40.7 & -1.10e-1 & 9.63e-6 &	  9.9 &	 25.7\\ 
 \hline
 120
 & 40.8 & -1.35e-1 & 3.13e-6 & 28.9 & 26.9
 & 5992.4 & -1.35e-1 & 4.48e-4 & 795.5 & 26.2
 & 44.1 & -1.32e-1 & 9.59e-6 & 21.3 &	 27.1 \\ 
 \hline
 150
 & 32.3 & -1.75e-1 & 3.95e-6 & 49.0 & 29.6
 & 4819.1  & -1.77e-1 & 9.41e-6 & 1502.9 & 29.9
 & 37.8  & -1.75e-1 & 9.26e-6 & 47.9 &	 29.3 \\ 
 \hline
\end{tabular}}
\end{table}
\section{Conclusions}\label{sec6} 

We proposed a practical inexact $q$-order regularized proximal Newton method for $q\in[2,3]$ to solve nonconvex composite problem \eqref{prob}, which becomes an inexact CR method for $q=3$. We proved the full convergence of the generated iterate sequence under the KL property of $F$, and the local superlinear rate of order $\frac{q-1}{\theta q}$ under the KL property of $F$ with exponent $\theta\in(0,\frac{q-1}{q})$. Furthermore, the iterate and objective value sequences were proved to have a local superlinear rate of order $\gamma(q\!-\!1)$ under a local H\"{o}lderian error bound of order $\gamma$ on the second-order stationary point set $\mathcal{X}^*$. These results improve those of exact or inexact CR methods obtained in \cite{Doikov20,Doikov22} for convex composite optimization,  and they are first extended to nonconvex composite problems. Numerical tests have shown that Algorithm \ref{qregPNT} armed with ZeroFPR as the inner solver is very promising for those problems with highly nonlinear $f$. An interesting topic is to develop a robust and fast inner solver, i.e., algorithm to solve the composite optimization problem of minimizing the sum of a $q$-order regularized nonconvex quadratic function and a nonsmooth convex function.

\section*{Acknowledgments}
The authors would like to thank the anonymous referees and the associated editor for their valuable suggestions and comments.

\section*{Data availability}
The data used in Sections \ref{sec5.3.1} and \ref{sec5.3.3} are available from LIBSVM (\url{https://www.csie.ntu.edu.tw/~cjlin/libsvmtools/datasets/}) and  yahoo finance (\url{https://finance.yahoo.com/}), respectively.

\section*{Declarations}
 The authors declare that they have no conflicts of interest.


\bibliographystyle{siamplain}
\bibliography{references}

\end{document}